\documentclass[12pt,reqno]{amsart}

\usepackage{amsmath,amsfonts,euscript,amscd,amsthm,amssymb,upref,graphics}
\usepackage[all]{xy}

\theoremstyle{definition}

\swapnumbers

\theoremstyle{plain}

\newtheorem{theorem}{Theorem}[section]
\newtheorem{proposition}[theorem]{Proposition}
\newtheorem{lemma}[theorem]{Lemma}
\newtheorem{corollary}[theorem]{Corollary}

\newtheorem{GizThm}[theorem]{Gizatullin's Theorem}

\theoremstyle{definition}

\newtheorem{definition}[theorem]{Definition}
\newtheorem{definitions}[theorem]{Definitions}
\newtheorem{parag}[theorem]{}

\newtheorem{notation}[theorem]{Notation}
\newtheorem{notations}[theorem]{Notations}

\newtheorem{remark}[theorem]{Remark}

\theoremstyle{remark}

\newtheorem*{smallremark}{Remark}

{\begin{enumerate}\setlength{\itemsep}{#1}}{\end{enumerate}}
\newenvironment{enumerata}%
{\begin{enumerate}

}{\end{enumerate}}
\newenvironment{Enumerata}[1]%
{\begin{enumerata}\setlength{\itemsep}{#1}}{\end{enumerata}}
{\begin{enumerate}
}{\end{enumerate}}

\newcommand{\supp}{	\operatorname{{\rm supp}}}

\newcommand{\ord}{	\operatorname{{\rm ord}}}

\newcommand{\Sing}{	\operatorname{{\rm Sing}}}

\newcommand{\Div}{	\operatorname{{\rm Div}}}

\newcommand{\Bs}{	\operatorname{{\rm Bs}}}
\newcommand{\dom}{	\operatorname{{\rm dom}}}
\newcommand{\codom}{	\operatorname{{\rm codom}}}
\renewcommand{\div}{	\operatorname{{\rm div}}}

\newcommand{\exc}{	\operatorname{{\rm exc}}}

\newcommand{\setspec}[2]{\big\{\,#1\, \mid \,#2\, \big\}}

\newcommand{\Integ}{\ensuremath{\mathbb{Z}}}
\newcommand{\Nat}{\ensuremath{\mathbb{N}}}

\newcommand{\Reals}{\ensuremath{\mathbb{R}}}
\newcommand{\aff}{\ensuremath{\mathbb{A}}}
\newcommand{\proj}{\ensuremath{\mathbb{P}}}
\newcommand{\bk}{{\ensuremath{\rm \bf k}}}

\newcommand{\nagata}{\ensuremath{\mathbb{F}\hspace{-.7pt}}}

\newcommand{\bbL}{\ensuremath{\mathbb{L}}}
\newcommand{\bbM}{\ensuremath{\mathbb{M}}}

\newcommand{\Aeul}{\EuScript{A}}
\newcommand{\Beul}{\EuScript{B}}

\newcommand{\Eeul}{\EuScript{E}}

\newcommand{\Geul}{\EuScript{G}}
\newcommand{\Heul}{\EuScript{H}}

\newcommand{\bup}{\leftarrow}

\newcommand{\temb}{\tilde\nu_{\text{\rm emb}}}

\newcommand{\isom}{\cong}
\renewcommand{\epsilon}{\varepsilon}
\renewcommand{\phi}{\varphi}
\renewcommand{\emptyset}{\varnothing}

\begin{document}
\renewcommand{\baselinestretch}{1.07}

\newcommand{\DAno}{\cite{Dai_Mell_1}}
\newcommand{\DA}[1]{\cite[#1]{Dai_Mell_1}}


\title[Linear systems associated to unicuspidal rational plane curves]
{Linear systems associated to\\ unicuspidal rational plane curves}

\author{Daniel Daigle, Alejandro Melle Hern\'andez}

\date{\today}

\address{Department of Mathematics and Statistics\\ University of Ottawa\\ Ottawa, Canada\ \ K1N 6N5}
\email{ddaigle@uottawa.ca}

\address{ICMAT (CSIC-UAM-UC3M-UCM) Dept.\ of Algebra, 
Facultad de Ciencias Matem\'aticas, Universidad Complutense, 28040, Madrid, Spain}
\email{amelle@mat.ucm.es} 

\thanks{Research of the first author supported by grants RGPIN/104976-2010 from NSERC Canada
and SAB2006-0060 from the Spanish Ministry of Education and Science.}
\thanks{Research of the second author supported by the grants MTM2010-21740-C02-01 and
Grupo Singular CCG07-UCM/ESP-2695-921020.}

{\renewcommand{\thefootnote}{}
}

\begin{abstract} 
Given a unicuspidal rational curve $C \subset \proj^2$ with singular point $P$, 
we study the unique pencil $\Lambda_C$ on $\proj^2$ satisfying
$C \in \Lambda_C$ and $\Bs(\Lambda_C)=\{P\}$.
We show that the general member of $\Lambda_C$ is a rational curve if and only if $\tilde\nu(C) \ge 0$,
where $\tilde\nu(C)$ denotes the self-intersection number of $C$ after the minimal resolution of
singularities.
We also show that if $\tilde\nu(C) \ge0$, then $\Lambda_C$ has a dicritical of degree $1$.
Note that all currently known unicuspidal rational curves $C \subset \proj^2$ satisfy $\tilde\nu(C) \ge0$.
\end{abstract}

\maketitle
 
\vfuzz=2pt


{\noindent\bf Introduction.} 

\medskip
A {\it unicuspidal rational curve\/} is a pair $(C,P)$ where $C$ is a curve
and $P\in C$ satisfies $C \setminus \{P\} \isom \aff^1$.  We call $P$ the distinguished point of $C$.

Let $C \subset \proj^2$ be a unicuspidal rational curve with distinguished point $P$.
In section 1 we define an infinite family of linear systems on $\proj^2$ determined by $(C,P)$ in a natural way.
We are particularly interested in two of these linear systems, denoted $\Lambda_C$ and $N_C$,
where $\Lambda_C$ is a pencil and $N_C$ is a net. 
In fact $\Lambda_C$ has the following characterization:
\begin{enumerate}

\item \it $\Lambda_C$ is the unique pencil on $\proj^2$ satisfying $C \in \Lambda_C$
and $\Bs( \Lambda_C ) = \{P\}$

\end{enumerate}
where $\Bs( \Lambda_C )$ denotes the base locus of $\Lambda_C$ on $\proj^2$.
The existence of this pencil was pointed out to us by A.\ Campillo and I.\ Luengo in a friendly conversation.
It appeared to us that it would be interesting to understand {\it how the properties of $C$ are related to
those of $\Lambda_C$}; this is the underlying theme of the present paper.

Given a curve $C \subset \proj^2$, let $\widetilde \proj^2 \to \proj^2$
be the minimal resolution of singularities of $C$
(this is the ``short'' resolution, not the ``embedded'' resolution);
let $\widetilde C \subset \widetilde \proj^2$ be the strict transform of $C$, 
and let ${\tilde \nu}(C)$ denote the self-intersection number of $\widetilde C$ on $\widetilde  \proj^2$.

For a unicuspidal rational curve $C \subset \proj^2$, we show
(cf.\ Theorems \ref{kjhggfqfdwfdqfssgiocxpoxzojcnnvb}, \ref{wedkkwejdewideideiiiiidss} and \ref{kjfkjeopozxpozx}):
\begin{enumerate}
\addtocounter{enumi}{1}

\item \it The general member of  $\Lambda_C$ is a rational curve if and only if $\tilde\nu(C) \ge 0$.

\item {\it The general member of  $N_C$ is a rational curve if and only if $\tilde\nu(C)>0$.}

\item \it If $\tilde\nu(C) \ge 0$ then $\Lambda_C$ has either $1$ or $2$ dicriticals,
and at least one of them has degree $1$.

\end{enumerate}

In view of these results, it is worth noting that
{\it all currently known unicuspidal rational curves $C \subset \proj^2$ satisfy $\tilde\nu(C) \ge 0$}. 
See \ref{8714rfhisfja} for details.

The proofs of the above statements (2) and (3) make use of results from \cite{Dai_Mell_1},
where we solved the following problem: given a curve $C$ on a rational nonsingular projective surface $S$,
find all linear systems $\bbL$ on $S$ satisfying $C \in \bbL$, $\dim\bbL \ge1$,
and the general member of $\bbL$ is a rational curve.

In statement (4) we claim, in particular, that
{\it if $\tilde\nu(C) \ge 0$ then $\Lambda_C$ has a dicritical of degree $1$} (see \ref{254desq3sa3235fa} for definitions).
It seems that the existence of such a dicritical is not an easy fact.
Indeed, the proof of this claim takes more than half of the present paper
(all of  sections \ref{SecAppendix} and \ref{ExistDicDeg1}). 
Note, however, that the graph theoretic tool developed in section \ref{SecAppendix} is 
susceptible of being useful in other settings.

\medskip
For a survey of open problems related to cuspidal rational plane curves,
the reader is referred to \cite{FLMN:Marseille2005}.

\bigskip
\noindent{\bf Acknowledgements.} 
The first author would like to thank the Department of Algebra of the Universidad Complutense de Madrid
for its hospitality.  The research which led to this article was initiated during a $7$ months stay of the first
author at that institution.

\bigskip
\noindent{\bf Conventions.} 
All algebraic varieties are over an algebraically closed field $\bk$ of characteristic zero.
Varieties (so in particular curves) are irreducible and reduced.
A divisor $D$ of a surface is {\it reduced\/} if $D = \sum_{i=1}^n C_i$ where $C_1, \dots, C_n$
are distinct curves ($n\ge0$). 
We write $e_Q(C)$ for the multiplicity of a point $Q$ on a curve $C$.

\section{Definition of $\Lambda_C$ and $N_C$}

A \textit{unicuspidal rational curve\/} is a pair $(C,P)$
where $C$ is a curve and $P$ is a point of $C$ such that $C \setminus \{P\} \isom \aff^1$.
We call $P$ the {\it distinguished point}, and we consider that the sentence
``$C$ is a unicuspical rational curve with distinguished point $P$''
is equivalent to ``$(C,P)$ is a unicuspical rational curve''.
We allow ourselves to speak of a 
unicuspidal rational curve $C$ without mentioning $P$,
but keep in mind that $C$ always comes equipped with a choice of a point $P$
(that choice being forced when $C \not\isom \proj^1$).

The aim of this section is to define,
given a unicuspidal rational curve $C \subset \proj^2$, 
an infinite family of linear systems $X_{\ell,j}(C)$ on $\proj^2$.
This is done in Proposition~\ref{qwqwoiwowiquwuyu}.
We are particularly interested in two of these linear systems, the pencil $\Lambda_C$ and the
net $N_C$, defined in \ref{dkfjasjdf;ajs;fa;o}--\ref{dkfjasjdf;ajs;fa;o-net}.

\begin{notations}  \label {378hfd83nf293}
Let $C \subset \proj^2$ be a unicuspical rational curve with distinguished point $P$.
If $D$ is an effective divisor in $\proj^2$, let 
$i_P(C, D)$ denote the local intersection number
of $C$ and $D$ at $P$ (which is defined to be $+\infty$ if
$C$ is a component of $D$).
Let $\Gamma = \Gamma_{(C,P)} \subseteq \Nat$ denote the semigroup of $(C,P)$,
i.e., the set of local intersection numbers
$i_P(C,D)$ where $D$ is an effective divisor such that
$C \not\subseteq \supp(D)$.
We also use the standard notation for intervals, $[a,b] = \setspec{ x \in \Reals }{ a \le x \le b }$.
\end{notations}

\begin{proposition}  \label {qwqwoiwowiquwuyu}
Let $C \subset \proj^2$ be a unicuspidal rational curve of degree $d$ and with distinguished point $P$.
For each pair $(\ell,j) \in \Nat^2$ such that $\ell>0$ and $j \le \ell d$,
let $X_{\ell,j}(C)$ be the set of effective divisors
$D$ of $\proj^2$ such that $\deg(D) = \ell$ and $i_P(C, D) \ge j$.
\begin{Enumerata}{1mm}

\item $X_{\ell,j}(C)$ is a linear system on $\proj^2$ for all $\ell,j$,
and $\dim X_{\ell,j}(C) \ge 1$ whenever $\ell \ge d$.

\item
For each $j\in\Nat$ such that $j\le d^2$,
the dimension of the linear system $X_{d,j}(C)$ is equal
to the cardinality of the set $[j,d^2] \cap \Gamma$,
where $\Gamma = \Gamma_{(C,P)}$.
In particular, for each integer $j$ such that $(d-1)(d-2) \le j \le d^2$,
$\dim X_{d,j}(C) = d^2-j+1$.
Consequently, $X_{d,d^2}(C)$ is a pencil and $X_{d,d^2-1}(C)$ is a net.

\end{Enumerata}
For each $\ell \in \Nat\setminus\{0\}$, define the abbreviation
$X_\ell(C) = X_{\ell,\ell d}(C)$.
Note that the above assertions imply that $X_d(C)$ is a pencil and that
$\dim X_\ell(C) \ge 1$ whenever $\ell \ge d$.
Moreover, if $\ell \in \Nat$ is such that $0< \ell < d$ then the following hold:
\begin{Enumerata}{1mm} \addtocounter{enumi}{2}

\item  \label {716161761716}
$X_\ell(C)$ contains at most one element
and if $X_\ell(C) \neq \emptyset$ then $\ell d \in \Gamma$.

\item 
$|\Gamma \cap [0, \ell d ]| \ge (\ell+1)(\ell+2)/2$,
and if equality holds and $\ell d \in \Gamma$ then $X_\ell(C) \neq \emptyset$.

\end{Enumerata}
\end{proposition}

\begin{smallremark}
The proof below is an elaboration of the proof
of Proposition~2 of \cite{FLMN:London2006};
moreover, the inequality in assertion~(d) is part of the cited result.
\end{smallremark}

\begin{smallremark}
$C \in X_{d,j}(C)$ for all $j$, because $i_P(C,C) = \infty > j$.
\end{smallremark}

\begin{proof}[Proof of \ref{qwqwoiwowiquwuyu}]
Choose coordinates $(X,Y,Z)$ for $\proj^2$ such that $P = (0:0:1)$.
Let $\bk[X,Y,Z]_\ell$ denote the vector space of homogeneous polynomials of 
degree $\ell$ and, given $G \in \bk[X,Y,Z]_\ell \setminus \{0\}$,
let $\div_0(G)$ be the effective divisor on $\proj^2$, of degree $\ell$,
with equation ``$G=0$''.
Let $F \in \bk[X,Y,Z]_d$ be an irreducible homogeneous 
polynomial of degree $d$ whose zero-set is $C$.
Let $x(t), y(t) \in t\bk[[t]]$ be a local parametrization of $C$ at $P$.
Then $F( x(t), y(t), 1 ) = 0$ and,
for any $\ell \in \Nat \setminus\{0\}$ and
$G \in \bk[X,Y,Z]_{ \ell } \setminus \{0\}$,
Bezout's Theorem gives
\begin{equation} \label{dkfjakjewf;kjaew}
\ord_t G( x(t), y(t), 1 )  = i_P(C , \div_0(G))
\begin{cases}
\in \Gamma \cap [0, \ell d],  &
\text{if $G \in \bk[X,Y,Z]_{ \ell } \setminus (F)$}, \\
= \infty, & \text{if $G \in \bk[X,Y,Z]_{ \ell } \cap (F)$}
\end{cases}
\end{equation}
where $(F)$ is the principal ideal of $\bk[X,Y,Z]$ generated by $F$.
Define a sequence of $\bk$-linear maps
$L_n : \bk[X,Y,Z] \to \bk$ (for $n \in \Nat$)
by the condition
$G( x(t), y(t), 1 ) = \sum_{n \in \Nat} L_n(G) t^n$ for any $G \in \bk[X,Y,Z]$.

Fix a pair $(\ell,j) \in \Nat^2$ such that $\ell \ge 1$ and $0 \le j \le \ell d$.
Consider the linear map of $\bk$-vector spaces
$$
T_\ell : \bk[X,Y,Z]_{\ell} \to \bk^{ | \Gamma \cap [0, \ell d ] | }, \quad
G \mapsto \big( L_{n_1}(G), \dots, L_{n_p}(G) \big),
$$
where $n_1 < \dots < n_p$ are the elements of $\Gamma \cap [0,\ell d]$, and define the subspace
$E_{\ell,j}$ of $\bk^{ | \Gamma \cap [0, \ell d ] | }$ by
$$
E_{\ell,j} =
\setspec{(0, \dots, 0,\lambda_1, \dots, \lambda_e)}
{ \lambda_1,\dots,\lambda_e \in \bk },
$$
where $e = | \Gamma \cap [j, \ell d] |$.
Note that \eqref{dkfjakjewf;kjaew} has the following two consequences:
firstly, $\ker T_\ell = \bk[X,Y,Z]_{ \ell } \cap (F)$, so
\begin{equation} \label{kfdjalfkkdkdkkdk}
\dim( \ker T_\ell ) = \begin{cases}
0, & \text{if $\ell < d$} \\
1, & \text{if $\ell = d$};
\end{cases}
\end{equation}
secondly,
\begin{multline*}
T_\ell^{-1}(E_{\ell,j}) \setminus \{0\}
= \setspec{ G \in \bk[X,Y,Z]_{ \ell } \setminus \{0\} }{ \ord_t G(x(t),y(t),1) \ge j } \\
= \setspec{ G \in \bk[X,Y,Z]_{ \ell } \setminus\{0\} }
{ i_P( C, \div_0(G) ) \ge  j },
\end{multline*}
so 
\begin{equation} \label{jii3i489234i3u498242jleih}
X_{\ell,j}(C)
= \setspec{ \div_0(G) }{ G \in T_\ell^{-1}(E_{\ell,j}) \setminus\{0\} } .
\end{equation}
In particular,
\begin{equation} \label{cvcoivocioijekjwekjwkejkwj}
\textit{$X_{\ell,j}(C)$ is a linear system of dimension }
\dim_\bk\big( T_\ell^{-1}(E_{\ell,j}) \big) - 1 .
\end{equation}
If $\ell \ge d$ then $\ker( T_\ell ) = \bk[X,Y,Z]_\ell \cap (F)$ has dimension equal to 
$\dim \bk[X,Y,Z]_{\ell - d} = \frac{(\ell-d)(\ell-d+3)}2 + 1$,
so 
$$
\dim X_{\ell,j}(C) = \dim T_\ell^{-1}(E_{\ell,j}) - 1 \ge \frac{(\ell-d)(\ell-d+3)}2.
$$
Hence, $\dim X_{\ell,j}(C) \ge 2$ whenever $\ell>d$, and 
$X_{\ell,j}(C) \neq \emptyset$ when $\ell=d$.  To finish the proof of assertion~(a),
we still need to show that  $\dim X_{\ell,j}(C) \ge 1$ when $\ell=d$. 

Consider the case $\ell = d$.
It is known that the number $\delta = (d-1)(d-2)/2$ satisfies
$2\delta + \Nat \subseteq \Gamma$ as well as
$\delta = | \Nat \setminus \Gamma |$.  As $2\delta < d^2$,
it follows that $d^2 + \Nat \subset \Gamma$ and 
$$
| \Gamma \cap [0, d^2] | = d^2+1-\delta = (d^2+3d)/2 = \dim_\bk \bk[X,Y,Z]_d - 1,
$$
so $\dim( \dom T_d ) = 1 + \dim( \codom T_d )$.
As $\dim( \ker T_d ) = 1$ by \eqref{kfdjalfkkdkdkkdk},
we obtain that $T_d$ is surjective and that (for any $j \le d^2$)
$\dim T_d^{-1}(E_{d,j}) = 1 + \dim E_{d,j} = 1 + | \Gamma \cap [j,d^2] |$,
so
\begin{equation} \label{dpiofupq39wejf}
\dim X_{d,j}(C) = | \Gamma \cap [j,d^2] |.
\end{equation}
As $d^2 \in \Gamma \cap [j,d^2]$, it follows in particular that $\dim X_{d,j}(C) \ge1$, which
finishes the proof of (a).
In the special case where $2\delta \le j \le d^2$ we have
$[j,d^2] \cap \Nat \subset \Gamma$, so \eqref{dpiofupq39wejf} gives
$$
\dim X_{d,j}(C) = d^2 - j + 1 .
$$
In particular $\dim X_{d,d^2}(C) = 1$ and $\dim X_{d,d^2-1}(C) = 2$,
so (b) is proved.

From now-on assume that $0 < \ell < d$.

Since $T_\ell$ is injective by \eqref{kfdjalfkkdkdkkdk}, and since
the definition of $E_{\ell, j}$ implies
\begin{equation} \label{difupq9wejf;laksduf}
\dim E_{\ell, \ell d} = | \Gamma \cap \{ \ell d \} | =
\begin{cases} 
1, & \text{if $\ell d \in \Gamma$,} \\
0, & \text{if $\ell d \notin \Gamma$,}
\end{cases} 
\end{equation}
we have $\dim T_\ell^{-1}(E_{\ell, \ell d}) \le 1$,
so \eqref{jii3i489234i3u498242jleih} implies that
$X_\ell(C) = X_{\ell,\ell d}(C)$ contains at most one element.
Moreover, if $X_\ell(C) \neq \emptyset$ then 
$\dim T_\ell^{-1}(E_{\ell, \ell d}) = 1$,
so $\dim E_{\ell, \ell d} = 1$ and \eqref{difupq9wejf;laksduf} implies
that $\ell d \in \Gamma$.
This proves (c).

To prove (d) note that the fact that $T_\ell$ is injective implies that
$\dim( \codom T_\ell ) \ge \dim( \dom T_\ell )$, i.e.,
\begin{equation} \label{ghqghqghqgqhgqhg}
|\Gamma \cap [0, \ell d ]| \ge (\ell+1)(\ell+2)/2.
\end{equation}
Suppose that equality holds in \eqref{ghqghqghqgqhgqhg};
then $T_\ell$ is bijective, and if we also assume that 
$\ell d \in \Gamma$ then
$\dim E_{\ell, \ell d} = 1$ by \eqref{difupq9wejf;laksduf},
so $T_\ell^{-1}( E_\ell )$ has dimension $1$ and
\eqref{jii3i489234i3u498242jleih} implies that 
$X_\ell(C) \neq \emptyset$.
This completes the proof of (d), and of the Proposition.
\end{proof}

\begin{definition} \label {dkfjasjdf;ajs;fa;o}
Let $C \subset \proj^2$ be a rational unicuspidal curve, with distinguished point $P$.
We define $\Lambda_C =  X_d(C)=X_{d,d^2}(C)$, where $d = \deg(C)$.
By \ref{qwqwoiwowiquwuyu}(b), $\Lambda_C$ is a pencil on $\proj^2$.
The definition of $X_{d,d^2}(C)$ and Bezout's Theorem yield the following explicit description of $\Lambda_C$:
$$
\Lambda_C = \{ C \} \cup
\setspec{D \in \Div( \proj^2 ) }{ \text{$D \ge 0$,\ \ $\deg(D) = \deg(C)$\ \ and\ \  $C \cap \supp(D) = \{P\}$} }.
$$
\end{definition}

The pencil $\Lambda_C$ can also be characterized as follows:

\begin{corollary} \label {dkjfslkdflks}
Let $C \subset \proj^2$ be a unicuspidal rational curve with distinguished point $P$.
Then $\Lambda_C$ is the unique pencil on $\proj^2$
satisfying $C \in \Lambda_C$ and $\Bs( \Lambda_C ) = \{P\}$.
\end{corollary}

\begin{proof}
From the explicit description of $\Lambda_C$ given in \ref{dkfjasjdf;ajs;fa;o},
it is clear that $C \in \Lambda_C$ and $\Bs( \Lambda_C ) = \{P\}$.
To prove uniqueness, consider a pencil $\Lambda$ on $\proj^2$
such that $C \in \Lambda$ and $\Bs( \Lambda ) = \{P\}$.
Let $D$ be any element of $\Lambda$ other than $C$.
Then (since $\Lambda$ is a pencil)
any point of $\supp(D) \cap C$ is in fact a base point
of $\Lambda$; so $\supp(D) \cap C = \{ P \}$.
Using again the explicit description of $\Lambda_C$ given in \ref{dkfjasjdf;ajs;fa;o},
this gives $D \in \Lambda_C$.
This shows that $\Lambda \subseteq \Lambda_C$ and hence that $\Lambda = \Lambda_C$.
\end{proof}

\begin{definition} \label {dkfjasjdf;ajs;fa;o-net}
Let $C \subset \proj^2$ be a rational unicuspidal curve, with distinguished point $P$.
Define $N_C = X_{d,d^2-1}(C)$, where $d = \deg(C)$. By \ref{qwqwoiwowiquwuyu}, $N_C$ is a net.
Observe that $\Lambda_C \subset N_C$ and that
$$
\Bs( N_C ) = \begin{cases}
\{P\}, & \text{if $\deg C > 1$,} \\
\emptyset, & \text{if $\deg C = 1$.}
\end{cases}
$$
Also note that the linear systems $\Lambda_C$ and $N_C$ are primitive
(i.e., their general member is irreducible and reduced), because
$C$ is irreducible and reduced and is 
an element of each of them.
\end{definition}

\begin{smallremark}
We shall restrict ourselves to studying the pencil $\Lambda_C$ and the net $N_C$ associated to
a unicuspidal rational curve $C \subset \proj^2$,
but the other linear systems defined in Proposition~\ref{qwqwoiwowiquwuyu} also deserve some attention.
For instance, consider the set 
$S_C = \setspec{ \ell \in \Nat }{ 0 \le \ell < d \text{ and } X_\ell(C) \neq \emptyset }$, where $d = \deg(C)$.
Parts (c) and (d) of the above proposition indicate that $S_C$ is closely related to the semigroup $\Gamma_{(C,P)}$,
and one can see that $S_C$ is also related to the reducible elements of $\Lambda_C$.
Something interesting can be said about these relations, but this theme is not developed in this paper.
\end{smallremark}

\begin{smallremark}
The objects
$X_{\ell,j}(C)$, $X_\ell(C)$, $\Lambda_C$ and $N_C$ should really be
denoted $X_{\ell,j}(C,P)$, $X_\ell(C,P)$, $\Lambda_{C,P}$ and $N_{C,P}$,
as they depend on the choice of $P$ in the nonsingular case.
\end{smallremark}


\section{Preliminaries on $\proj^1$-rulings on rational surfaces}

In this section, $S$ is a rational nonsingular projective surface.

\begin{definition}  \label {kd545342100ko}
A pencil $\Lambda$ on $S$ is called a \textit{$\proj^1$-ruling\/}
if it is base-point-free and if its general member is isomorphic to $\proj^1$.
If $\Lambda$ is a $\proj^1$-ruling of $S$ then by a \textit{section\/} of
$\Lambda$ we mean an irreducible curve $\Sigma \subset S$
such that $\Sigma \cdot D =1$ for any $D \in \Lambda$ (it then follows
that $\Sigma \isom \proj^1$). 
\end{definition}

The following is a well-known consequence of the Riemann-Roch Theorem for $S$:

\begin{lemma} \label {dskjfwejf;akj}
If $C \subset S$ satisfies $C \isom \proj^1$ and $C^2 = 0$ then the
complete linear system $|C|$ on $S$ is a $\proj^1$-ruling.
\end{lemma}

\begin{parag} \label {sdewffrefcd3}
Recall that, given $k \in \Nat$, there exists a triple
$(\nagata_k, \bbL_k, \Delta_k )$ where
$\nagata_k$ is a nonsingular projective rational surface,
$\bbL_k$ is a base-point-free pencil on $\nagata_k$
each of whose elements is a projective line,
and $\Delta_k$ is a section of $\bbL_k$ satisfying $\Delta_k^2 = -k$.
Moreover, $(\nagata_k, \bbL_k, \Delta_k )$ is uniquely determined by $k$
up to isomorphism.
The surface $\nagata_k$ is called the Nagata-Hirzebruch ruled surface
of degree $k$.
\end{parag}

\begin{parag}  \label{difupawejk}
By an {\it SNC-divisor\/} of $S$ we mean a divisor $D = \sum_{i=1}^n C_i$
where $C_1, \dots, C_n$ ($n \ge0$) are distinct curves on $S$ and:
\begin{itemize}

\item each $C_i$ is a nonsingular curve;

\item for every choice of $i \neq j$ such that $C_i \cap C_j \neq \emptyset$,
$C_i \cap C_j$ is one point and the local intersection number
of $C_i$ and $C_j$ at that point is equal to $1$;

\item if $i,j,k$ are distinct then $C_i \cap C_j \cap C_k = \emptyset$.

\end{itemize}
The {\it dual graph\/} of an SNC-divisor $D=\sum_{i=1}^n C_i$ of $S$
is the weighted graph defined by stipulating that the vertex set is $\{ C_1, \dots, C_n \}$,
that distinct vertices $C_i$, $C_j$ are joined by an edge if
and only if $C_i \cap C_j \neq \emptyset$, and that the weight of the vertex $C_i$
is the self-intersection number $C_i^2$.
\end{parag}

For the following fact, see for instance \cite[Chap.~2, 2.2]{Miy_RatUnirat} or  \cite[Sec.~2]{Giz70}.

\begin{GizThm} \label{jhdhsgghgaqaqoaiaaqo}
Let $\Lambda$ be a $\proj^1$-ruling on $S$.
Then $\Lambda$ has a section and the following hold:
\begin{enumerata}

\item Let $D \in \Lambda$.
Then each irreducible component of $D$ is
a projective line and $\supp(D)$ is the support of an SNC-divisor of $S$
whose dual graph is a tree.
If $\supp(D)$ is irreducible then $D$ is reduced.
If $\supp(D)$ is reducible then there exists a $(-1)$-component $\Gamma$
of $\supp(D)$ which meets at most two other components of $\supp(D)$;
moreover, if $\Gamma$ has multiplicity $1$ in the divisor $D$ then
there exists another $(-1)$-component
of $\supp(D)$ which meets at most two other components of $\supp(D)$.

\item Let $\Sigma$ be a section of $\Lambda$.
Then there exist a nonsingular projective surface $\nagata$ and
a birational morphism $\rho : S \to \nagata$ satisfying:
\begin{itemize}

\item the exceptional locus of $\rho$ is the union of
the irreducible curves $C \subset S$
which are $\Lambda$-vertical\footnote{A curve $C \subset S$ is said to be
\textit{$\Lambda$-vertical\/} if it is included in
the support of an element of $\Lambda$.}
and disjoint from $\Sigma$;

\item the linear system $\bbL = \rho_*( \Lambda )$ is a base-point-free
pencil on $\nagata$ each of whose elements is a projective line,
and the curve $\Delta = \rho( \Sigma )$ is a section of $\bbL$;

\item $\nagata = \nagata_k$ for some $k \in \Nat$;
moreover, if $\Sigma^2 \le 0$ then $\Sigma^2 = -k$ and
$(\nagata, \bbL, \Delta ) = (\nagata_k, \bbL_k, \Delta_k )$.

\end{itemize}
\end{enumerata}
\end{GizThm}

\section{Rational linear systems; uniresolvable curves and linear systems}

We continue to assume that $S$ is a rational nonsingular projective surface.

\begin{definition}  \label {983r9daweifhado}
We say that a linear system $\bbL$ on $S$ is {\it rational\/}
if $\dim\bbL \ge 1$ and the general member of $\bbL$ is an irreducible rational curve.
\end{definition}

\begin{notations} \label{minimalresolsing}
Given a curve $C \subset S$,
consider the minimal resolution of singularities $X \to S$ of $C$,
let $\tilde C$ be the strict transform of $C$ on $X$,
and let $\tilde\nu(C)$ denote the self-intersection number of $\tilde C$ in $X$.
When $\tilde\nu(C) \ge 0$ (resp.\ $\tilde\nu(C)>0$), we say that $C$ is {\it of nonnegative type\/}
(resp.\ {\it of positive type}).
We also consider the embedded  resolution of singularities $Y \to S$ of $C$,
and define $\temb(C)$ to be the self-intersection number of the strict transform of $C$ on $Y$.
Clearly, $\temb(C) \le \tilde\nu(C)$.
\end{notations}

\medskip
Let $C \subset S$ be a curve.
It follows from \DA{Theorem 2.8} that
the existence of a rational pencil $\Lambda$ on $S$ satisfying $C \in \Lambda$ 
is equivalent to $C$ being rational and of nonnegative type.
Let us now be more precise in the special case where $C$ is ``uniresolvable''.

\begin{definitions} \label {ehwgehgwhepcopdopsdopsdo}
Consider a sequence
\begin{equation} \label {xixoixiixixuxixoix}
S = S_0 \xleftarrow{\ \pi_1\ } S_1 \xleftarrow{\ \pi_2\ } \cdots
\xleftarrow{\ \pi_n\ } S_{n}
\end{equation}
where $\pi_i : S_i \to S_{i-1}$ is the blowing-up of the nonsingular
surface $S_{i-1}$ at a point $P_i \in S_{i-1}$ (for $1 \le i \le n$).
\begin{enumerata}

\item We say that the sequence \eqref{xixoixiixixuxixoix} is a 
\textit{chain\/} if $\pi_{i-1}(P_i) = P_{i-1}$ for all $i$ such that $2\le i \le n$.

\item A linear system $\bbL$ on $S$ is \textit{uniresolvable\/}
if $\dim\bbL\ge1$, $\bbL$ is without fixed components
and there exists a  chain \eqref{xixoixiixixuxixoix} with the property
that the strict transform of $\bbL$ on $S_n$ is base-point-free.

\item A curve $C \subset S$ is \textit{uniresolvable\/}
if there exists a chain \eqref{xixoixiixixuxixoix}
with the property that the strict transform of $C$ on $S_n$
is a nonsingular curve.

\end{enumerata}
\end{definitions}

Note that if $C \subset S$ is uniresolvable then there exists at least one point $P \in C$
such that $\Sing(C) \subseteq \{P\}$.

\begin{theorem}  \label {dkfjalskdjfias777}
Let $C \subset S$ be a uniresolvable curve and let $P \in C$ be such that $\Sing(C) \subseteq \{P\}$.
Then the following are equivalent:
\begin{enumerata}

\item $C$ is rational and of nonnegative type;

\item there exists a rational linear system $\bbL$ on $S$ satisfying $C \in \bbL$;

\item there exists a rational and uniresolvable pencil $\Lambda$
on $S$ such that $C \in \Lambda$ and $\Bs(\Lambda) \subseteq \{P\}$.

\end{enumerata}
\end{theorem}

\begin{proof}
It follows from \DA{Theorem 2.8} that (a) is equivalent to (b), and it is clear that (c) implies (b);
so it suffices to prove that (a) implies (c).  Assume that (a) is satisfied.
Then there exists a chain \eqref{xixoixiixixuxixoix} satisfying:
\begin{itemize}

\item the strict transform $C_n \subset S_n$ of $C$ is nonsingular and satisfies $C_n^2=0$;

\item $P_1=P$ and, for each  $i\ge2$, $P_i$ lies on the strict transform $C_{i-1} \subset S_{i-1}$ of $C$.

\end{itemize}
By \ref{dskjfwejf;akj}, $|C_n|$ is a $\proj^1$-ruling on $S_n$.
Define $\Lambda = \pi_* | C_n |$, where $\pi = \pi_1 \circ \cdots \circ \pi_n : S_n \to S_0$.
Then $\Lambda$ is a rational pencil on $S$ satisfying $C\in\Lambda$.
The strict transform of $\Lambda$ on $S_n$ is $|C_n|$, which is base-point-free.
This has two consequences: (i) all infinitely near base points of $\Lambda$ are among $\{P_1, \dots, P_n\}$,
so in particular $\Bs(\Lambda) \subseteq \{P\}$; (ii) since \eqref{xixoixiixixuxixoix} is a chain,
$\Lambda$ is uniresolvable.
\end{proof}

Let us also mention the following related fact:

\begin{lemma} \label{goigoigoigoieawwaw}
Let $\Lambda$ be a pencil on $S$ and $C \subset S$ an irreducible component
of the support of some member of $\Lambda$.
If $\Lambda$ is rational and uniresolvable, then $C$ is rational and uniresolvable.
\end{lemma}

\begin{proof}

Consider the minimal resolution \eqref{xixoixiixixuxixoix} of the base points of $\Lambda$;
since $\Lambda$ is uniresolvable, \eqref{xixoixiixixuxixoix} is a chain.
Let $\Lambda_n$ (resp.\ $C_n$) be the strict transform of $\Lambda$ (resp.\ of $C$) on $S_n$.
As $\Lambda$ is rational,
the general member of $\Lambda_n$ is isomorphic to $\proj^1$,
so $\Lambda_n$ is a $\proj^1$-ruling.
As $C_n$ is included in the support of some element of $\Lambda_n$,
Gizatullin's Theorem~\ref{jhdhsgghgaqaqoaiaaqo}
implies that $C_n$ is nonsingular and rational.
So $C$ is rational and 
(since \eqref{xixoixiixixuxixoix} is a chain) uniresolvable.
\end{proof}

\section{Rationality of $\Lambda_C$ and $N_C$}

Given a unicuspidal rational curve $C \subset \proj^2$
we consider the pencil $\Lambda_C$ 
and the net $N_C$ defined in  \ref{dkfjasjdf;ajs;fa;o},
and ask when these linear systems are rational (in the sense of \ref{983r9daweifhado}).

\begin{theorem} \label {kjhggfqfdwfdqfssgiocxpoxzojcnnvb}
For a unicuspidal rational curve $C \subset \proj^2$, the following are equivalent:
\begin{enumerata}

\item $C$ is of nonnegative type

\item $\Lambda_C$ is rational.

\end{enumerata}
Moreover, if these conditions hold then $\Lambda_C$ is uniresolvable.
\end{theorem}

\begin{proof}
The fact that (b) implies (a) follows from either one of \DA{2.8} or \ref{dkfjalskdjfias777}.
Conversely, suppose that (a) holds and let $P$ be the distinguished point of $C$.
Then, in particular, $C$ is uniresolvable and $P \in C$ is such that $\Sing(C) \subseteq \{P\}$.
By \ref{dkfjalskdjfias777},
there exists a rational and uniresolvable pencil $\Lambda$
on $\proj^2$ such that $C \in \Lambda$ and $\Bs(\Lambda) \subseteq \{P\}$;
then $\Bs(\Lambda) = \{P\}$.
By \ref{dkjfslkdflks}, $\Lambda_C$ is the unique pencil on $\proj^2$
satisfying $C \in \Lambda_C$ and $\Bs( \Lambda_C ) = \{P\}$.
Thus $\Lambda=\Lambda_C$.
Consequently, $\Lambda_C$ is rational and uniresolvable.
\end{proof}

\begin{remark}  \label {8714rfhisfja}
In view of \ref{kjhggfqfdwfdqfssgiocxpoxzojcnnvb}, it is interesting to note:
\begin{enumerata}

\item \it All unicuspidal rational curves $C \subset \proj^2$ satisfying 
$\bar\kappa( \proj^2 \setminus C ) < 2$ are of nonnegative type.

\item \it All currently known unicuspidal rational curves $C \subset \proj^2$ are of nonnegative type.

\end{enumerata}
Indeed, let $C \subset \proj^2$ be a unicuspidal rational curve and consider 
$\bar\kappa = \bar\kappa( \proj^2 \setminus C )$, the logarithmic Kodaira dimension of 
$\proj^2 \setminus C$.
Then it is a priori clear that $\bar\kappa \in \{ -\infty, 0, 1, 2\}$.
\begin{itemize}

\item If $\bar\kappa = -\infty$ then \cite{MiS} implies that $\temb(C) \ge-1$,
and it follows that $\tilde\nu(C)>0$.

\item The case $\bar{\kappa}=0$ cannot occur by a result of Tsunoda \cite{Tsu1}.

\item The case $\bar{\kappa}=1$ is completely classified in \cite{KeitaTono}, and the multiplicity
sequences are given explicitly.
A straightforward computation using these sequences shows that $\tilde\nu(C) \in \{0,1\}$, where the two cases occur.

\item The case $\bar{\kappa}=2$ is not classified. 
The only known examples here are two families of curves found by Orevkov in \cite{Or}.
For these examples the multiplicity sequences are known explicitly, and 
a straightforward computation shows that $\tilde\nu(C) \in \{1,4\}$ where the two cases occur.

\end{itemize}
One should also remark that the sets
\begin{gather*}
\setspec{ \tilde\nu(C) }{ C \subset \proj^2 \text{ cuspidal rational } }\\
\setspec{ \temb(C) }{ C \subset \proj^2 \text{ unicuspidal rational, $\bar\kappa(\proj^2 \setminus C)=1$ } }
\end{gather*}
are not bounded below, as can be deduced from \cite{Fenske1} and \cite{KeitaTono}, respectively. 
\end{remark}

The next paragraph will be used as a reference, when we want
to establish the notation:

\begin{notations}  \label  {sdopfiuqpweijdfklsmdn897823487}
Let $C \subset \proj^2$ be a unicuspidal rational curve with distinguished point $P$.
Then $(C,P)$ determines an infinite sequence
\begin{equation}  \label{diufpq3poaksdjfi8888}
\proj^2 = S_0 \xleftarrow{\ \pi_1\ } S_1 \xleftarrow{\ \pi_2\ }
S_2 \xleftarrow{\ \pi_3\ } \cdots
\end{equation}
of nonsingular projective surfaces and blowing-up morphisms such that,
for each $i\ge1$,
$\pi_i : S_i \to S_{i-1}$ is the blowing-up
of $S_{i-1}$ at the unique point $P_i \in S_{i-1}$ which lies on
the strict transform of $C$ and which is mapped to $P_1=P$
by $\pi_1 \circ \cdots \circ \pi_{i-1} : S_{i-1} \to S_0$.
Let $E_i = \pi_i^{-1}(P_i) \subset S_i$ and, if $i<j$,
let the strict transform of $E_i$ on $S_j$ be also denoted by $E_i \subset S_j$.
Let $C_i \subset S_i$ be the strict transform of $C_0 = C$ on $S_i$,
and let $\Lambda_i$ be the strict transform of $\Lambda_0 = \Lambda_C$ on $S_i$.
By definition of the sequence \eqref{diufpq3poaksdjfi8888}, it is clear that
\begin{equation} \label{dedsedwsewdwededswewd}
\textit{$C_{i-1} \cap E_{i-1} = \{ P_i \}$ in $S_{i-1}$,\ \ for all $i\ge2$.}
\end{equation}
Let $n \le N$ be the natural numbers satisfying:
\begin{itemize}

\item $S_n \to S_0$ is the minimal resolution of singularities of $C$;

\item $S_N \to S_0$ is the minimal embedded resolution of singularities of $C$.

\end{itemize}
Finally, let $r_i = e_{P_i}(C_{i-1})$ for all $i \ge1$, and let $d=\deg(C)$.
Then the invariants $\tilde\nu(C)$ and $\temb(C)$
defined in \ref{minimalresolsing} are given by
$$
\tilde\nu(C) = C_n^2 = d^2-\sum_{i=1}^n r_i^2
\quad \text{and} \quad  \temb(C) = C_N^2.
$$
It is clear that if $C$ is singular then $N=n+r_n$ and hence
 $\temb(C)= \tilde\nu(C) - r_n$, and that
if $C$ is nonsingular (i.e., $d \le 2$)
then $N=n=0$ and $\temb(C)= \tilde\nu(C)=d^2$.
\end{notations}  

\begin{smallremark}
If $\tilde\nu(C)\ge0$, the natural number $m$ defined in \ref{dif4938998r984j20w2j} (below)
is to be added to the set of notations introduced in \ref{sdopfiuqpweijdfklsmdn897823487}. 
Note that the inequality $n \le \min(N,m)$ always holds,
and that the three cases $m<N$, $m=N$ and $m>N$ can occur.
\end{smallremark}

\begin{proposition}  \label {dif4938998r984j20w2j}
Let $C \subset \proj^2$ be a unicuspidal rational curve with distinguished point $P$, and let the
notation be as in \ref{sdopfiuqpweijdfklsmdn897823487}.
If $C$ is of nonnegative type, then the following hold.
\begin{Enumerata}{1mm}

\item  There exists a natural number $m \ge n$ such that 
$S_m \to S_0$ is the minimal resolution of the base points of $\Lambda_C$.

\item  \label {fdfdffdfdfd5g}
$C_i \in \Lambda_i$ for all $i \in \{ 0, \dots,  m \}$.

\item $\Lambda_m$ is a $\proj^1$-ruling of $S_m$ (cf.~\ref{kd545342100ko}). 

\item $C_m \isom \proj^1$ and $C_m^2 = 0$.

\item \label {ghghghghtytr}
For all $i \in \{ 1, \dots, m \}$, the following hold in $S_m$:
$$
\textit{$E_i$ is horizontal $\iff E_i \cap C_m \neq \emptyset \iff P_{m+1} \in E_i$.}
$$
Here we say that a curve in $S_m$ is \emph{vertical} if it is included in the
support of a member of $\Lambda_m$, and \emph{horizontal} if it is not vertical.

\item \label {hghshsghsghshgs}
$E_m$ is horizontal and at most one $i < m$ is such that $E_i \subset S_m$ is horizontal.

\item \label{positive-char}
$E_m$ is a section of $\Lambda_m$ if and only if $C$ is of positive type.

\end{Enumerata}
\end{proposition}

\begin{proof}
Let $S = Y_0 \xleftarrow{\ \rho_1\ } Y_1 \xleftarrow{\ \rho_2\ } \cdots \xleftarrow{\ \rho_m\ } Y_{m}$
be the minimal resolution of the base points of $\Lambda_C$,
where, for $1 \le i \le m$, $\rho_i : Y_i \to Y_{i-1}$ is the blowing-up of the nonsingular
surface $Y_{i-1}$ at a point $P^*_i \in Y_{i-1}$.
As $C$ is of nonnegative type,
\ref{kjhggfqfdwfdqfssgiocxpoxzojcnnvb} implies that $\Lambda_C$ is rational.
Let $\widetilde C \subset Y_m$ (resp.\ $\widetilde \Lambda_C$) be the strict transform of $C$
(resp.\ of $\Lambda_C$) on $Y_m$.
By \DA{2.7(b)}, the fact that $\Lambda_C$ is rational implies that 
$\widetilde C \in \widetilde \Lambda_C$ and that $\widetilde C$ is nonsingular.
From $\widetilde C \in \widetilde \Lambda_C$, we deduce that for each $i$ the base point $P_i^*$ lies on
the strict transform of $C$ on $Y_{i-1}$; as $P_i^*$ is infinitely near $P$ (because $\Bs( \Lambda_C ) = \{P\}$),
it follows that $(P_1^*, \dots, P_m^*) = (P_1, \dots, P_m)$.
Thus $S_m \to S_0$ is the minimal resolution of the base points of $\Lambda_C$.
As we have observed, $\widetilde C = C_m$ is nonsingular; it follows that $m \ge n$, so (a) is proved.

Then $\Bs( \Lambda_{i-1} ) = \{ P_i \}$ for all $i \in \{1, \dots, m \}$, and $\Bs( \Lambda_{m} ) = \emptyset$.

We already noted that 
$\widetilde C \in \widetilde \Lambda_C$, which we may rewrite as $C_m \in \Lambda_m$. 
It follows that assertion~(b) holds.
As $\Lambda_C$ is a rational pencil, so is $\Lambda_m$;
as $\Lambda_m$ is base-point-free, its general member is a $\proj^1$, so (c) holds.
By $C_m \in \Lambda_m$ and $\Bs(\Lambda_m) = \emptyset$, we get $C_m^2=0$, so
assertion~(d) holds.

The fact that $C_m \in \Lambda_m$
and that $\Lambda_m$ is base-point-free implies that if $C' \subset S_m$ is a
curve distinct from $C_m$ then $C'$ is horizontal if and only if
$C' \cap C_m \neq \emptyset$.  In particular, (e) is proved, and (f) immediately follows.

To prove (g), note that
$E_m$ is a section of $\Lambda_m$ if and only if $E_m\cdot C_m=1$,
if and only if $C_{m-1}$ is nonsingular;
as $C_{m-1}^2 > C_m^2=0$, this is equivalent to $C$ being of positive type.
\end{proof}

\begin{theorem}    \label{wedkkwejdewideideiiiiidss}
For a unicuspidal rational curve $C \subset \proj^2$,
the following are equivalent:
\begin{Enumerata}{1mm}

\item $C$ is of positive type;

\item $N_C$ is rational;

\item the rational map $\Phi_{N_C} : \proj^2 \dasharrow \proj^2$, corresponding
to the net $N_C$, is birational.

\end{Enumerata}
Moreover, if the above conditions hold then
the Cremona map $\Phi_{N_C}$ transforms $C$ into a line,
and $\Lambda_C$ into a pencil of ``all lines through some point''.
\end{theorem}

\begin{proof}
The fact that (c) implies (b) is trivial.
If (b) holds then parts (e) and (f) of \DA{2.8} imply that $N_C \subseteq \bbL_C$
and that $\dim\bbL_C = \tilde\nu(C)+1$, so $\tilde\nu(C) > 0$, showing that (b) implies (a).
There remains to show that if (a) holds then 
$\Phi_{N_C}$ is birational and transforms $C$ into a line and
$\Lambda_C$ into a pencil of all lines through some point.

Suppose that $C$ is of positive type and let the
notation be as in~\ref{sdopfiuqpweijdfklsmdn897823487} and \ref{dif4938998r984j20w2j}.
By \ref{dif4938998r984j20w2j}\eqref{positive-char}, $E_m$ is a section of $\Lambda_m$.
Then Gizatullin's Theorem~\ref{jhdhsgghgaqaqoaiaaqo} implies that there exists a birational
morphism $\rho : S_m \to \nagata_1$ whose exceptional locus $\exc(\rho) \subset S_m$ 
is a union of $\Lambda_m$-vertical curves in $S_m$ and $\exc(\rho) \cap E_m = \emptyset$.
Moreover, in the notation of~\ref{sdewffrefcd3}, $\rho_*( \Lambda_m )$ is the standard ruling $\bbL_1$
of $\nagata_1$ and $\rho(E_m)$ is the $(-1)$-section of that ruling.
As the exceptional locii of the two morphisms
$S_{m-1} \xleftarrow{ \pi_{m} } S_m \xrightarrow{ \rho } \nagata_1$
are disjoint, we have the commutative diagram
$$
\xymatrix{
\proj^2 = S_0 \ar @{<-} [r]^{\pi_1} & \ \cdots\ 
& S_{m-1}  \ar @{<-} [r]^{\pi_m} \ar [l]_{\pi_{m-1}} \ar[d]^{\bar\rho}
&  S_m \ar[d]^{\rho}  \\
& & \proj^2  &  \nagata_1 \ar[l]_{\bar\pi_m}
}
$$
where $\bar\pi_m : \nagata_1 \to \proj^2$ is the contraction of $\rho(E_m)$.
Define the birational map \mbox{$\Phi : \proj^2 \dashrightarrow \proj^2$} as
the composition
$$
\xymatrix{
S_0 \ar @{.>}[rr]^{ (\pi')^{-1} } && S_{m-1} \ar[r]^{ \bar\rho } & \proj^2,
}
$$
where $\pi' = \pi_1 \circ \cdots \circ \pi_{m-1} : S_{m-1} \to S_0$.
It is clear that $\Phi$ transforms $C$ into a line in $\proj^2$.
Also, $\Phi$ determines a net $N$ on $\proj^2$ (without fixed components);
let us show that $N = N_C$.

Consider the group homomorphisms
$$
\Div( S_0 ) \xleftarrow{\ \pi'_*\ } \Div( S_{m-1} )
\xleftarrow{\ \bar\rho^*\ } \Div( \proj^2 )
$$
where $\bar\rho^*$ is the operation of taking the total transform with respect
to $\bar\rho$ and $\pi'_*$ takes direct image with respect to $\pi'$.
Let $Q = \bar\rho( P_m ) \in \proj^2$ and
let $\bbL$ be the linear system on $\proj^2$ consisting of 
all lines through $Q$.
Then the strict transform of $\bbL$ on $S_{m-1}$ (via $\bar \rho$)
is $\Lambda_{m-1}$.
As $\bar\rho$ restricts
to an isomorphism from a neighborhood of $P_m$ to a neighborhood of $Q$
(because $\exc(\rho) \cap E_m = \emptyset$),
the strict transform of $\bbL$ coincides with the total transform of $\bbL$,
so $\bar\rho^*$ transforms $\bbL$ into $\Lambda_{m-1}$ and consequently
$\pi'_* \circ \bar\rho^*$ transforms $\bbL$ into $\Lambda_C$.
Now we note that 
$\pi'_* \circ \bar\rho^*$ transforms $\bbM$ into $N$,
where $\bbM$ is the linear system of all lines in $\proj^2$.
As $\bbL \subset \bbM$, it follows that $\Lambda_C \subset N$ (in particular
the elements of $N$ have degree $d = \deg C$).

Let $\bbM^\circ$ be the set of $M \in \bbM$ such that
$Q \notin M$ and $\bar\rho^{-1}(M)$ is an irreducible curve in $S_{m-1}$.
Then the image of 
$\bbM^\circ$ via $\pi'_* \circ \bar\rho^*$ is a dense subset of $N$.
Since $N$ and $N_C$ have the same dimension, in order to show
that $N=N_C$ it suffices to show that
$\pi'_* \circ \bar\rho^*$ maps $\bbM^\circ$ into $N_C$.
Let $M \in \bbM^\circ$ and consider the curve
$D = (\pi'_* \circ \bar\rho^*)(M) = \pi'( \bar\rho^{-1}(M) ) \subset S_0$.

Let $L = \bar\rho( C_{m-1} ) \in \bbL$ and note that 
$\bar \rho$ restricts to an isomorphism
from a neighbourhood of $C_{m-1}$ to a neighbourhood of $L$.
As $(M \cdot L)_{\proj^2} =1$ and the point $M \cap L$ is not $Q$,
it follows that 
$$
(\bar\rho^{-1}(M) \cdot C_{m-1})_{ S_{m-1} } = 1
$$
and that the point 
$\bar\rho^{-1}(M) \cap C_{m-1} = \{R\}$ belongs to $C_{m-1}\setminus P_{m}$,
so $R \notin \exc( \pi' )$.
Consequently, $D \cap C \subseteq \{ \pi'(R), P \}$ and
$i_{\pi'(R)}(D , C) =1$, where the point $\pi'(R)$ is distinct from $P$.
By Bezout, $i_P(D , C) = d^2-1$, so $D \in N_C$.
This shows that $\pi'_* \circ \bar\rho^*$ maps $\bbM^\circ$ into $N_C$;
it follows that $N=N_C$, as desired.

So $\Phi_{N_C} = \Phi$ and consequently $\Phi_{N_C}$ is birational.
We already noted that $\Phi$ transforms $C$ into a line and that 
$\pi'_* \circ \bar\rho^*$ transforms $\bbL$ into $\Lambda_C$, so the last assertions follow.
\end{proof}

\section{Intermezzo: erasable weighted pairs}
\label {SecAppendix}

The aim of this section is to prove Proposition~\ref{656758765876dfkasbkjgfakh67664},
which is needed in the proof of Theorem~\ref{kjfkjeopozxpozx}.
Our proof of Proposition~\ref{656758765876dfkasbkjgfakh67664} 
makes use of a theory of ``erasable weighted pairs'' which we develop in this section;
in fact Proposition~\ref{bvbwvwbwvbevbrbtobvbvp} is the only fact from this
graph theory which is needed,
but its proof requires several preliminary lemmas.

We stress that the present section is completely self-contained.
Except for the fact that 
Proposition~\ref{656758765876dfkasbkjgfakh67664} is used in the proof of Theorem~\ref{kjfkjeopozxpozx},
this section is completely independent from the rest of the paper.

Our graphs have finitely many vertices and edges, edges are not directed, 
no edge connects a vertex to itself, and at most one edge exists between a given pair of vertices.  
A {\it weighted graph\/} is a graph in which each vertex is assigned an integer (called the weight of the vertex).
Note that the empty graph is a weighted graph.
We assume that the reader is familiar with the classical notion of blowing-up of a weighted graph,
and refer to 1.1 and 1.2 of \cite{Dai:Chains} for details.
In particular, recall that there are three ways to blow-up a weighted graph $\Geul$:
one can blow-up $\Geul$ {\it at a vertex}, or {\it at an edge}, or one can perform the {\it free blowing-up\/} of $\Geul$
(in the last case, one takes the disjoint union of $\Geul$ and of a vertex of weight $-1$).
In all cases, blowing-up $\Geul$ produces a new weighted graph $\Geul'$
whose vertex-set is obtained from that of $\Geul$ by adding one new vertex $e$ of weight $-1$ (one says that
$e$ is the vertex ``created'' by the blowing-up).
If $\Geul'$ is a blowing-up of $\Geul$ and $e$ is the vertex of $\Geul'$  created by the blowing-up,
then one says that $\Geul$ is the blowing-down of $\Geul'$ at $e$.
Two weighted graphs $\Aeul$ and $\Beul$ are {\it equivalent\/} (denoted $\Aeul \sim \Beul$)
if one can be obtained from the other by a finite sequence of blowings-up and
blowings-down.
Note that if $\Geul$ is a weighted graph without edges, and in which each vertex has weight $-1$,
then $\Geul$ is equivalent to the empty weighted graph $\emptyset$.

\begin{definitions} \label{dfj;wieufajfklaj;dklfj;89868}
\mbox{\ }
\begin{enumerata}

\item By a \textit{weighted pair,} we mean an ordered pair $(\Geul, v)$
where $\Geul$ is a nonempty weighted graph and $v$ is a vertex of $\Geul$
(called the distinguished vertex).

\item
A \textit{blowing-up\/} of a weighted pair $(\Geul, v)$
is a weighted pair $(\Geul', v')$ satisfying:
\begin{itemize}

\item the weighted graph $\Geul'$ is obtained by blowing-up the 
weighted graph $\Geul$ either at the vertex $v$ or at an edge incident to $v$

\item $v'$ is the unique vertex of $\Geul'$ which is not a vertex of $\Geul$
(i.e., $v'$ is the vertex of weight $-1$ which is created by the 
blowing-up).

\end{itemize}
We write $(\Geul, v) \bup (\Geul', v')$ to indicate that 
$(\Geul', v')$ is a blowing-up of $(\Geul, v)$.

\item A weighted pair $(\Geul,v)$ is said to be \textit{erasable\/}
if there exists a finite sequence
\begin{equation} \label{37463786457823645897}
(\Geul,v) = (\Geul_0, e_0) \bup (\Geul_1, e_1) \bup \cdots \bup (\Geul_n, e_n)
\qquad \text{(with $n \ge 0$)}
\end{equation}
of blowings-up of weighted pairs such that
the weighted graph $\Geul_n \setminus \{ e_n \}$ is equivalent
to the empty weighted graph.

\end{enumerata}
\end{definitions}

\begin{remark}
In contrast with the theory of weighted graphs,
we do not define a ``blowing-down'' of weighted pairs.
The contraction of weighted pairs defined in
\ref{453724982736hsdjhasdgy8734978jkdfh} is not the inverse 
operation of the blowing-up of weighted pairs.
\end{remark}

\begin{remark} \label{5d656d50df90d9f0s=-f=00ijk}
Let $(\Geul,v)$ be a weighted pair.
The following claims are obvious:
\begin{enumerata}

\item {\it If $\Geul$ has a vertex $w$ of nonnegative weight
such that $w \neq v$ and $w$ is not a neighbor of $v$, then 
$(\Geul,v)$ is not erasable.}

\item {\it If $\Geul$ has at least two vertices, $v$ has negative weight
and all weights in $\Geul \setminus \{v\}$ are strictly less than $-1$,
then $(\Geul,v)$ is not erasable.}

\end{enumerata}
\end{remark}

\begin{definition}
For any weighted pair $(\Geul,v)$ we define
$\ell(\Geul,v) \in \Nat \cup \{ \infty \}$ as follows.
If $(\Geul,v)$ is not erasable, we set $\ell(\Geul,v) = \infty$.
If $(\Geul,v)$ is erasable, then we define $\ell(\Geul,v)$ to be the least
$n \in \Nat$ for which there exists a sequence~\eqref{37463786457823645897}
satisfying $\Geul_n \setminus \{ e_n \} \sim \emptyset$.
Thus a weighted pair $(\Geul,v)$ is erasable if and only if 
$\ell(\Geul,v) < \infty$.
Also note that the condition $\ell(\Geul,v) = 0$ is equivalent to
$\Geul \setminus \{v\} \sim \emptyset$.
\end{definition}

\begin{definition}
Let $(\Geul,v)$ be an erasable weighted pair such that $\ell(\Geul,v) > 0$.
A blowing-up $(\Geul',v')$ of $(\Geul,v)$ is said to be \textit{good\/} if
it satisfies $\ell (\Geul',v') < \ell (\Geul,v)$.
Note that if $(\Geul',v')$ is a good blowing-up of $(\Geul,v)$
then $(\Geul',v')$ is erasable and 
$\ell (\Geul',v') = \ell (\Geul,v) - 1$.
\end{definition}

\begin{lemma} \label {dkfjqwejpqoiweropi[wepi}
If $(\Geul,v)$ is an erasable weighted pair such that $\ell(\Geul,v) > 0$,
then there exists a good blowing-up of $(\Geul,v)$.
\end{lemma}

\begin{proof}
Since $(\Geul,v)$ is erasable,
there exists a sequence~\eqref{37463786457823645897} such that
$\Geul_n \setminus \{ e_n \} \sim \emptyset$ and $n = \ell(\Geul,v)$.
By assumption we have $n>0$, so we may consider the blowing-up
$(\Geul,v) \bup (\Geul_1,e_1)$.
Then $\ell (\Geul_1,e_1) = \ell (\Geul,v) - 1$, so 
$(\Geul_1,e_1)$ is a good blowing-up of $(\Geul,v)$.
\end{proof}

\begin{definition} \label{453724982736hsdjhasdgy8734978jkdfh}
Let $(\Geul,v)$ be a weighted pair.
A \textit{contractible vertex\/} of $(\Geul,v)$ is a vertex
$w$ of $\Geul$ satisfying:
\begin{itemize}

\item the weight of $w$ is $(-1)$ and
$w$ has at most two neighbours in $\Geul$

\item if $w$ has two neighbours $v_1$ and $v_2$,
then $v_1,v_2$ are not joined by an edge

\item $w \neq v$ and $w$ is not a neighbour of $v$.

\end{itemize}
If $w$ is a contractible vertex of $(\Geul,v)$ then the 
\textit{contraction of $(\Geul,v)$ at $w$} is the weighted pair
$(\bar\Geul, \bar v)$ defined by taking $\bar \Geul$ to be the 
blowing-down of the weighted graph $\Geul$ at $w$ and by setting
$\bar v = v$.
\end{definition}

\begin{lemma} \label{7r6d7g6dfgfuerhju475387}
Suppose that $(\bar\Geul, \bar v)$  is the contraction of
a weighted pair $(\Geul,v)$ at some contractible vertex.
Then $\ell(\Geul,v) = \ell(\bar\Geul, \bar v)$.
\end{lemma}

\begin{proof}
We proceed by induction on $n = \min( \ell(\Geul,v) , \ell(\bar\Geul, \bar v) )$, noting that the lemma is true whenever $n = \infty$.
Let $w$ be the contractible vertex of $(\Geul, v)$ at which the contraction is performed.
Then $\bar \Geul \setminus \{ \bar v \}$ is the blowing-down of $\Geul \setminus \{ v \}$ at $w$,
so there is an equivalence of weighted graphs $\Geul \setminus \{ v \} \sim \bar \Geul \setminus \{ \bar v \}$.
In particular, the lemma is true whenever $n=0$.

Consider $n \in \Nat \setminus \{0\}$ such that the lemma is true for all 
$(\Geul,v)$ and $(\bar\Geul, \bar v)$ satisfying $\min( \ell(\Geul,v) , \ell(\bar\Geul, \bar v) ) < n$. 
Consider $(\Geul,v)$ and $(\bar\Geul, \bar v)$ such that $\min( \ell(\Geul,v) , \ell(\bar\Geul, \bar v) ) = n$. 

Choose an element $(\Geul_0, v_0)$ of the set $\big\{ (\Geul,v) , (\bar\Geul, \bar v) \big\}$ such that 
$\ell(\Geul_0,v_0) = n$, and let $(\Geul_0', v_0')$ denote the other element of the set.
By \ref{dkfjqwejpqoiweropi[wepi},
there exists a blowing-up $(\Geul_0,v_0) \bup (\Geul_1, v_1)$ such that $\ell(\Geul_1,v_1) = n-1$.
Then $\Geul_1$ is the blowing-up of $\Geul_0$ at $x$, where $x$ is either the distinguished vertex $v_0$ or an 
edge $\{ v_0, u \}$ with $u$ a neighbour of $v_0$ in $\Geul_0$.
As the distinguished vertices of
$(\Geul_0,v_0)$ and $(\Geul_0', v_0')$ are the same ($v_0 = v_0'$ because $v=\bar v$),
and the neighbours of that vertex are the same
in $\Geul_0$ and in $\Geul_0'$, it makes sense to blow-up $\Geul_0'$ at $x$, 
and this gives rise to a blowing-up $(\Geul_0', v_0') \bup (\Geul_1', v_1')$ of weighted pairs.
Let us change the notation again and represent the two blowings-up
$(\Geul_0,v_0) \bup (\Geul_1, v_1)$ and $(\Geul_0', v_0') \bup (\Geul_1', v_1')$ 
as
$$
(\Geul,v) \bup (\Heul, e) \quad \text{and} \quad (\bar \Geul, \bar v) \bup (\bar \Heul, \bar e) \qquad \text{(in some order)}.
$$
Note that $w$ is a contractible vertex of $(\Heul, e)$,
and that $(\bar \Heul, \bar e)$ is the contraction of $(\Heul, e)$ at $w$.
We have
$$
\min( \ell( \Heul, e ), \ell( \bar \Heul, \bar e ) )
= \min( \ell(\Geul_1, v_1), \ell(\Geul_1', v_1') )
\le \ell(\Geul_1, v_1) = n-1,
$$
so the inductive hypothesis implies that  $\ell( \Heul, e ) = \ell( \bar \Heul, \bar e )$, which is equal to $n-1$.
Thus $\ell(\Geul,v) \le 1 + \ell(\Heul,e) =n$ and $\ell(\bar \Geul,\bar v) \le 1 + \ell(\bar \Heul,\bar e) =n$,
so 
$$
\max( \ell( \Geul, v ), \ell( \bar \Geul, \bar v ) ) \le n = \min( \ell( \Geul, v ), \ell( \bar \Geul, \bar v ) )
$$
and consequently $\ell( \Geul, v ) = \ell( \bar \Geul, \bar v ) )$.
\end{proof}

\begin{notation}
Given integers $x_1, \dots, x_n$ and $i \in \{1, \dots, n\}$,
the weighted pair
$$
\setlength{\unitlength}{1mm}
\begin{picture}(66,6)(-3,-3)
\put(0,0){\circle*{1}}
\put(20,0){\circle*{1}}
\put(30,0){\circle*{1}}
\put(40,0){\circle*{1}}
\put(60,0){\circle*{1}}
\put(0,0){\line(1,0){7}}
\put(20,0){\line(-1,0){7}}
\put(10,0){\makebox(0,0){\dots}}
\put(50,0){\makebox(0,0){\dots}}
\put(20,0){\line(1,0){27}}
\put(60,0){\line(-1,0){7}}
\put(0,-1){\makebox(0,0)[t]{\tiny $x_1$}}
\put(20,-1){\makebox(0,0)[t]{\tiny $x_{i-1}$}}
\put(30,-1){\makebox(0,0)[t]{\tiny $x_{i}$}}
\put(40,-1){\makebox(0,0)[t]{\tiny $x_{i+1}$}}
\put(60,-1){\makebox(0,0)[t]{\tiny $x_n$}}
\put(30,1){\makebox(0,0)[b]{\tiny $*$}}
\end{picture}
$$
(where the asterisk $*$ indicates the distinguished vertex)
is denoted by 
$$
[x_1, \dots, x_{i-1}, x_i^*, x_{i+1}, \dots, x_n].
$$
Observe that there is an equality of weighted pairs
$$
[x_1, \dots, x_{i-1}, x_i^*, x_{i+1}, \dots, x_n]
=
[x_n, \dots, x_{i+1}, x_i^*, x_{i-1}, \dots, x_1].
$$
\end{notation}

\begin{smallremark}
The next result asserts that a certain implication is true.
We will find later that this implication has a false hypothesis---which, of course, is not a problem.
\end{smallremark}

\begin{lemma}  \label{5657df69a78d09a87se<F7>k4l;234l5k2;o}
$\ell [-2,-1^*, -1, -3] < \infty\ \implies\ 
\ell [-3,-1^*, -1, -2] < \ell[-2,-1^*, -1, -3]$.
\end{lemma}

\begin{proof}
Suppose that $(\Geul,v) = [-2,-1^*, -1, -3]$ is erasable 
and observe that $\ell(\Geul,v) > 0$.
Pick a sequence \eqref{37463786457823645897}
such that $\Geul_n \setminus \{e_n\} \sim \emptyset$ and such that 
$n = \ell(\Geul,v)$.
Then $(\Geul_1, e_1)$ is a good blowing-up of $(\Geul,v)$ and one of
the following holds:
\begin{enumerata}

\item $(\Geul_1, e_1)$ is the blowing-up of $(\Geul,v)$ at $v$

\item $(\Geul_1, e_1)$ is the blowing-up of $(\Geul,v)$ at 
the edge $[-1^*, -1]$

\item $(\Geul_1, e_1)$ is the blowing-up of $(\Geul,v)$ at 
the edge $[-2,-1^*]$.

\end{enumerata}
In case~(a), one of the connected components of $\Geul_n \setminus \{ e_n \}$ has the form
\begin{equation} \label{34gf3g4f3gf4g3f4g3f4g3f}
\setlength{\unitlength}{1mm}
\raisebox{-7\unitlength}{\begin{picture}(36,16)(-3,-3)
\put(0,0){\circle*{1}}
\put(10,0){\circle*{1}}
\put(20,0){\circle*{1}}
\put(30,0){\circle*{1}}
\put(0,0){\line(1,0){30}}
\put(10,0){\line(0,1){5}}
\put(0,-1){\makebox(0,0)[t]{\tiny $-2$}}
\put(10,-1.5){\makebox(0,0)[t]{\tiny $z$}}
\put(20,-1){\makebox(0,0)[t]{\tiny $-1$}}
\put(30,-1){\makebox(0,0)[t]{\tiny $-3$}}
\put(12,9){\makebox(0,0)[l]{\tiny $\Beul$}}
\put(10,9){\oval(3,10)}
\end{picture}} \qquad \text{(for some $z \in \Integ$)}
\end{equation}
where every vertex in the branch $\Beul$ has weight strictly less than $-1$
(and $\Beul$ might be empty).
This is absurd, because the weighted
graph \eqref{34gf3g4f3gf4g3f4g3f4g3f} is not equivalent to $\emptyset$.
Thus case~(a) does not occur.

In case~(b),
\ref{5d656d50df90d9f0s=-f=00ijk}(b) implies that
$(\Geul_1, e_1)$ is not erasable, which is absurd; 
so case~(b) does not occur either.

In case~(c) we have $(\Geul_1, e_1) = [-3, -1^*, -2, -1, -3]$,
and the contraction of 
$(\Geul_1, e_1)$ at its contractible vertex is
$(\bar \Geul_1, \bar e_1) = [-3, -1^*, -1, -2]$. Consequently
$$
\ell [-3, -1^*, -1, -2] = \ell (\bar \Geul_1, \bar e_1)
= \ell (\Geul_1, e_1) < \ell (\Geul, v) = \ell [-2,-1^*, -1, -3]
$$
(where we used \ref{7r6d7g6dfgfuerhju475387}),
and this proves the lemma.
\end{proof}

\begin{lemma} \label{cvnbcvnbx,nmv.z,mnv,js746597645782}
If $x \le -2$ and $\ell[-1^*, -1, x, -4] < \infty$, then
$$
\ell[-1^*, -1, x, -4] > \ell[ -3, -1^*, -1, -2] .
$$
\end{lemma}

\begin{proof}
Let $x \le -2$, let $(\Geul,v) = [-1^*, -1, x, -4]$ and suppose
that $\ell (\Geul,v) < \infty$.
As $\ell (\Geul,v) > 0$, there exists a good blowing-up 
$(\Geul',v')$ of $(\Geul,v)$. By \ref{5d656d50df90d9f0s=-f=00ijk}(b),
$(\Geul',v')$ cannot be the blowing-up of $(\Geul,v)$ at the
edge $[-1^*, -1]$; so 
$(\Geul',v')$ is the blowing-up of $(\Geul,v)$ at $v$, i.e.,
$(\Geul',v') = [-1^*, -2, -1, x, -4]$. The contraction of
$(\Geul',v')$ at its contractible vertex is
$(\bar\Geul',\bar v') = [-1^*, -1, x+1, -4]$, so 
$$
\ell [-1^*, -1, x+1, -4] 
= \ell (\bar\Geul',\bar v') 
= \ell (\Geul',v') 
< \ell (\Geul,v)
= \ell[-1^*, -1, x, -4].
$$
More precisely, we have shown that
if $x \le -2$ and $\ell[-1^*, -1, x, -4] < \infty$ then
$$
\ell[-1^*, -1, x, -4] > \ell [-1^*, -1, x+1, -4].
$$
By induction it follows that 
if $x \le -2$ and $\ell[-1^*, -1, x, -4] < \infty$, then
$$
\ell[-1^*, -1, x, -4] > \ell [-1^*, -1, -1, -4] = \ell [-1^*, 0, -3]
$$
(where the equality follows from \ref{7r6d7g6dfgfuerhju475387});
so there only remains to show that 
\begin{equation} \label{dfpqoi23u9873489rioe}
\ell [-1^*, 0, -3] \ge  \ell[ -3, -1^*, -1, -2] .
\end{equation}
This is obvious if 
$\ell [-1^*, 0, -3] = \infty$, so let us assume that
$\ell [-1^*, 0, -3] < \infty$.
Let $(\Geul,v) = [-1^*, 0, -3]$.
As $\ell (\Geul,v) > 0$, there exists a good blowing-up 
$(\Geul',v')$ of $(\Geul,v)$. By \ref{5d656d50df90d9f0s=-f=00ijk}(a),
$(\Geul',v')$ cannot be the blowing-up of $(\Geul,v)$ at $v$,
so it must be the blowing-up of $(\Geul,v)$ at the edge $[-1^*, 0]$; so 
$(\Geul',v') = [-2, -1^*, -1, -3]$ and consequently
\begin{equation} \label{4yt3yt42yut42u3y765}
\ell [-2, -1^*, -1, -3] = \ell (\Geul',v') < \ell (\Geul,v)
= \ell  [-1^*, 0, -3] < \infty.
\end{equation}
As $\ell [-2, -1^*, -1, -3] < \infty$,
\ref{5657df69a78d09a87se<F7>k4l;234l5k2;o} implies that 
$\ell [-3,-1^*, -1, -2] < \ell[-2,-1^*, -1, -3]$,
so \eqref{4yt3yt42yut42u3y765} gives
$$
\ell [-3,-1^*, -1, -2] < \ell[-2,-1^*, -1, -3] < \ell  [-1^*, 0, -3].
$$
So \eqref{dfpqoi23u9873489rioe} is proved and we are done.
\end{proof}

\begin{lemma} \label{cvgdvcdgfgcvgvcsiyuegtou78}
If $x \le -2$ and $\ell[-1, -1^*, x, -4] < \infty$, then
$$
\ell[-1, -1^*, x, -4] > \ell[ -3, -1^*, -1, -2] .
$$
\end{lemma}

\begin{proof}
Let $E$ be the set of $x \in \Integ$ satisfying $x \le -2$ and
\begin{equation} \label{dfjpqio3u4rp934rjuf}
\ell[-1, -1^*, x, -4] < \infty
\textit{\ \ and\ \ }
\ell[-1, -1^*, x, -4] \le \ell[ -3, -1^*, -1, -2] .
\end{equation}
It suffices to show that $E = \emptyset$.
By contradiction, suppose that $E \neq \emptyset$ and pick $x \in E$.
Let $(\Geul,v) = [-1, -1^*, x, -4]$.
Then $\ell (\Geul,v) < \infty$ and $\ell (\Geul,v) > 0$,
so there exists a good blowing-up $(\Geul',v')$ of $(\Geul,v)$.
By \ref{5d656d50df90d9f0s=-f=00ijk}(b),
$(\Geul',v')$ cannot be the blowing-up of $(\Geul,v)$ at
the edge $[-1, -1^*]$; so one of the following conditions must hold:
\begin{enumerata}

\item $(\Geul',v')$ is the blowing-up of $(\Geul,v)$ at $v$

\item $(\Geul',v')$ is the blowing-up of $(\Geul,v)$ at
the edge $[-1^*, x]$.

\end{enumerata}
In  case~(a), the contraction of $(\Geul',v')$ at its contractible
vertex is 
$$
(\bar \Geul',\bar v') = [-1^*, -1, x, -4].
$$
Thus $\ell [-1^*, -1, x, -4] = \ell (\bar \Geul',\bar v')
= \ell (\Geul',v') < \ell (\Geul,v) < \infty$,
so \ref{cvnbcvnbx,nmv.z,mnv,js746597645782}
implies that $\ell[-1^*, -1, x, -4] > \ell[ -3, -1^*, -1, -2]$.
This gives
$$
\ell[ -3, -1^*, -1, -2] < \ell[-1^*, -1, x, -4] < \ell (\Geul,v)
= \ell [-1, -1^*, x, -4],
$$
which contradicts \eqref{dfjpqio3u4rp934rjuf}
(and \eqref{dfjpqio3u4rp934rjuf} holds since $x \in E$).
Thus case~(a) does not occur.

In case~(b), $(\Geul',v') = [-1, -2, -1^*, x-1, -4]$.
The contraction of $(\Geul',v')$ at its contractible
vertex is $(\bar \Geul',\bar v') = [-1, -1^*, x-1, -4]$, so
$\ell [-1, -1^*, x-1, -4] = \ell (\bar \Geul',\bar v')
= \ell (\Geul',v') < \ell (\Geul,v) = \ell[-1, -1^*, x, -4]$.
In fact we have shown:
$$
\textit{if $x \in E$ then 
$\ell [-1, -1^*, x-1, -4] < \ell[-1, -1^*, x, -4]$ and $x-1 \in E$.}
$$
This implication together with $E \neq \emptyset$ imply the existence
of an infinite descending sequence
$$
\ell[-1, -1^*, x, -4] > \ell [-1, -1^*, x-1, -4] > \ell [-1, -1^*, x-2, -4]
> \dots
$$
of natural numbers, which is absurd.
So $E = \emptyset$ and we are done.
\end{proof}

\begin{lemma} \label{g3gg4gg6g8gg9ggg3g3g}
$[-3,-1^*, -1, -2]$ is not erasable.
\end{lemma}

\begin{proof}
We prove this by contradiction.
Let $(\Geul_0,e_0) = [-3,-1^*, -1, -2]$ and assume that
$(\Geul_0,e_0)$ is erasable.
As $\ell(\Geul_0,e_0)>0$, there exists a good blowing-up
$(\Geul_1,e_1)$ of $(\Geul_0,e_0)$.
There are three possibilities:
\begin{enumerata}

\item $(\Geul_1,e_1)$ is the blowing-up of $(\Geul_0,e_0)$ at $e_0$
\item $(\Geul_1,e_1)$ is the blowing-up of $(\Geul_0,e_0)$ at
the edge $[-1^*, -1]$
\item $(\Geul_1,e_1)$ is the blowing-up of $(\Geul_0,e_0)$ at
the edge $[-3,-1^*]$.

\end{enumerata}

Consider case~(a).
Let $(\bar \Geul_1, \bar e_1)$ be obtained from $(\Geul_1,e_1)$ by performing
two contractions at contractible vertices. Then
$(\bar \Geul_1, \bar e_1) = [ -1^*, 0, -3]$, so
$\ell(\bar \Geul_1, \bar e_1)>0$, so 
$(\bar \Geul_1, \bar e_1)$ has a good blowing-up $(\bar \Geul_2, \bar e_2)$.
By \ref{5d656d50df90d9f0s=-f=00ijk}(a),
the blowing-up of $(\bar \Geul_1, \bar e_1)$ at $\bar e_1$ is not good;
so $(\bar \Geul_2, \bar e_2)$ must be the blowing-up
of $(\bar \Geul_1, \bar e_1)$ at the edge $[ -1^*, 0]$,
i.e., $(\bar \Geul_2, \bar e_2) = [-2, -1^*, -1, -3]$.
Then
\begin{multline*}
\ell [-2, -1^*, -1, -3] = \ell (\bar \Geul_2, \bar e_2)
< \ell (\bar \Geul_1, \bar e_1)
= \ell (\Geul_1, e_1)< \ell (\Geul_0, e_0) \\
= \ell [-3,-1^*, -1, -2],
\end{multline*}
so $\ell [-2, -1^*, -1, -3] < \ell [-3,-1^*, -1, -2] < \infty$,
which contradicts \ref{5657df69a78d09a87se<F7>k4l;234l5k2;o}.
So case~(a) cannot occur.

In case~(b) we have $(\Geul_1,e_1) = [-3,-2,-1^*, -2, -2]$,
which is not erasable by \ref{5d656d50df90d9f0s=-f=00ijk}(b).
So case~(b) does not occur either.

In case~(c) we have $(\Geul_1,e_1) = [-4, -1^*, -2, -1, -2]$.
Let $(\bar \Geul_1, \bar e_1)$ be obtained from $(\Geul_1,e_1)$ by performing
two contractions at contractible vertices. Then
$(\bar \Geul_1, \bar e_1) = [ -4, -1^*, 0]$, so
$\ell(\bar \Geul_1, \bar e_1)>0$, so 
$(\bar \Geul_1, \bar e_1)$ has a good blowing-up $(\bar \Geul_2, \bar e_2)$.
In fact $(\bar \Geul_2, \bar e_2)$ must be the blowing-up of 
$(\bar \Geul_1, \bar e_1)$ at the edge $[ -1^*, 0]$, otherwise
\ref{5d656d50df90d9f0s=-f=00ijk}(a) gives a contradiction.
So $(\bar \Geul_2, \bar e_2) = [-4, -2, -1^*, -1] = [-1, -1^*, -2, -4]$
and consequently
$$
\ell [-1, -1^*, -2, -4] = \ell (\bar \Geul_2, \bar e_2)
< \ell (\bar \Geul_1, \bar e_1) = \ell (\Geul_1, e_1)
< \ell (\Geul_0, e_0)
= \ell [-3,-1^*, -1, -2].
$$
We conclude that
$$
\ell [-1, -1^*, -2, -4] < \ell [-3,-1^*, -1, -2] < \infty,
$$
which contradicts \ref{cvgdvcdgfgcvgvcsiyuegtou78}. So we are done.
\end{proof}

\begin{proposition} \label{bvbwvwbwvbevbrbtobvbvp}
Let $x \in \Integ \setminus \{-2\}$ and $y \in \Integ$.
Then the two weighted pairs
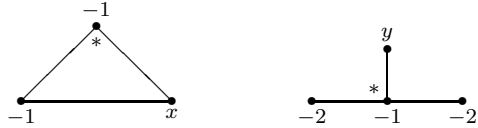
\begin{equation*} 
\setlength{\unitlength}{1mm}
\begin{picture}(26,17)(-3,-3)
\put(0,0){\circle*{1}}
\put(20,0){\circle*{1}}
\put(10,10){\circle*{1}}
\put(10,8.5){\makebox(0,0)[t]{\tiny $*$}}
\put(0,0){\line(1,0){20}}
\put(0,0){\line(1,1){10}}
\put(10,10){\line(1,-1){10}}
\put(0,-1){\makebox(0,0)[t]{\tiny $-1$}}
\put(20,-1){\makebox(0,0)[t]{\tiny $x$}}
\put(10,11){\makebox(0,0)[b]{\tiny $-1$}}
\end{picture}
\qquad\quad
\begin{picture}(26,13)(-3,-3)
\put(0,0){\circle*{1}}
\put(10,0){\circle*{1}}
\put(20,0){\circle*{1}}
\put(10,7){\circle*{1}}
\put(0,0){\line(1,0){20}}
\put(10,0){\line(0,1){7}}
\put(0,-1){\makebox(0,0)[t]{\tiny $-2$}}
\put(10,-1){\makebox(0,0)[t]{\tiny $-1$}}
\put(9,1){\makebox(0,0)[br]{\tiny $*$}}
\put(20,-1){\makebox(0,0)[t]{\tiny $-2$}}
\put(10,8){\makebox(0,0)[b]{\tiny $y$}}
\end{picture}
\end{equation*}
are not erasable.
\end{proposition}

\begin{proof}
Let $(\Geul,v)$ be the weighted pair which looks like a triangle, in the 
statement of the proposition, and (proceeding by contradiction)
assume that $(\Geul,v)$ is erasable.
Since $x \neq -2$, we have $\Geul\setminus\{v\} \not\sim \emptyset$,
so $\ell(\Geul,v) > 0$.  Pick a sequence \eqref{37463786457823645897}
such that $\Geul_n \setminus \{e_n\} \sim \emptyset$ and such that 
$n = \ell(\Geul,v)$;
note that $(\Geul_1, e_1)$ is a good blowing-up of $(\Geul,v)$.
If $(\Geul_1, e_1)$ is the blowing-up of $(\Geul,v)$ at $v$
then $\Geul_n \setminus \{e_n\}$ contains a simple circuit, which 
contradicts $\Geul_n \setminus \{e_n\} \sim \emptyset$;
so $(\Geul_1, e_1)$ is the blowing-up of $(\Geul,v)$
at one of the two edges incident to $v$.
Consequently, $(\Geul_1,e_1)$ is either as in \eqref{dkfjq;kwje;iqj;w}
or as in \eqref{8508ufjdht893475}, below.

Consider the case where $(\Geul_1,e_1)$ is as follows:
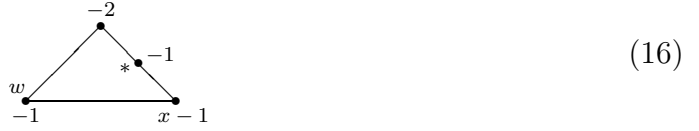
\begin{equation} \label{dkfjq;kwje;iqj;w}
\setlength{\unitlength}{1mm}
\raisebox{-8\unitlength}{\begin{picture}(26,17)(-3,-3)
\put(0,0){\circle*{1}}
\put(20,0){\circle*{1}}
\put(10,10){\circle*{1}}
\put(0,0){\line(1,0){20}}
\put(0,0){\line(1,1){10}}
\put(10,10){\line(1,-1){10}}
\put(0,-1){\makebox(0,0)[t]{\tiny $-1$}}
\put(21,-1){\makebox(0,0)[t]{\tiny $x-1$}}
\put(10,11){\makebox(0,0)[b]{\tiny $-2$}}
\put(15,5){\circle*{1}}
\put(14,5){\makebox(0,0)[tr]{\tiny $*$}}
\put(16,5){\makebox(0,0)[bl]{\tiny $-1$}}
\put(0,1){\makebox(0,0)[br]{\tiny $w$}}
\end{picture}}
\end{equation}
Then $w$ is a contractible vertex and if $(\bar \Geul_1,\bar e_1)$  denotes
the contraction of $(\Geul_1,e_1)$ at $w$ then 
$(\bar \Geul_1,\bar e_1)$ is isomorphic\footnote{The definition of
\textit{isomorphism of weighted pairs\/} is the obvious one.}
to $(\Geul,v)$.
This isomorphism implies that $\ell (\bar \Geul_1,\bar e_1)
= \ell (\Geul,v)$ but on the other hand \ref{7r6d7g6dfgfuerhju475387}
implies that 
$\ell (\bar \Geul_1,\bar e_1) = \ell (\Geul_1,e_1) < \ell (\Geul,v)$.
This contradiction shows that $(\Geul_1,e_1)$ cannot be as in 
\eqref{dkfjq;kwje;iqj;w}.

The only other possibility is that $(\Geul_1,e_1)$ be as follows:
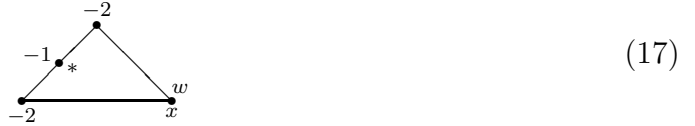
\begin{equation} \label{8508ufjdht893475}
\setlength{\unitlength}{1mm}
\raisebox{-8\unitlength}{\begin{picture}(26,17)(-3,-3)
\put(0,0){\circle*{1}}
\put(20,0){\circle*{1}}
\put(10,10){\circle*{1}}
\put(6,5){\makebox(0,0)[tl]{\tiny $*$}}
\put(0,0){\line(1,0){20}}
\put(0,0){\line(1,1){10}}
\put(10,10){\line(1,-1){10}}
\put(0,-1){\makebox(0,0)[t]{\tiny $-2$}}
\put(20,-1){\makebox(0,0)[t]{\tiny $x$}}
\put(10,11){\makebox(0,0)[b]{\tiny $-2$}}
\put(5,5){\circle*{1}}
\put(4,5){\makebox(0,0)[br]{\tiny $-1$}}
\put(20,1){\makebox(0,0)[bl]{\tiny $w$}}
\end{picture}}
\end{equation}
Now we must have $x=-1$, otherwise $\Geul_n \setminus \{e_n\}$ would not
contain any vertex of weight $(-1)$ and hence would not be equivalent
to the empty weighted graph.
So $w$ is a contractible vertex and 
the contraction $(\bar \Geul_1,\bar e_1)$  of $(\Geul_1,e_1)$ at $w$
is isomorphic to $(\Geul,v)$.
This leads to the same contradiction as in the first case, so
we have shown that $(\Geul,v)$ is not erasable.

\bigskip
From now-on let $(\Geul,v)$ be the weighted pair on the right-hand-side,
in the statement of the proposition; proceeding again by contradiction,
assume that $(\Geul,v)$ is erasable.
It is clear that $\Geul\setminus\{v\} \not\sim \emptyset$,
so $\ell(\Geul,v) > 0$.  Pick a sequence \eqref{37463786457823645897}
such that $\Geul_n \setminus \{e_n\} \sim \emptyset$ and such that 
$n = \ell(\Geul,v)$, and note that $(\Geul_1, e_1)$ is a good
blowing-up of $(\Geul,v)$.
One of the following holds:
\begin{enumerata}

\item $(\Geul_1, e_1)$ is the blowing-up of $(\Geul, v)$
at the edge which contains the vertex of weight $y$

\item $(\Geul_1, e_1)$ is the blowing-up of $(\Geul, v)$
at the distinguished vertex $v$

\item $(\Geul_1, e_1)$ is the blowing-up of $(\Geul, v)$
at an edge which does not contain the vertex of weight $y$.

\end{enumerata}

In case~(a), one of the connected components of $\Geul_n \setminus \{e_n\}$ has the following shape:
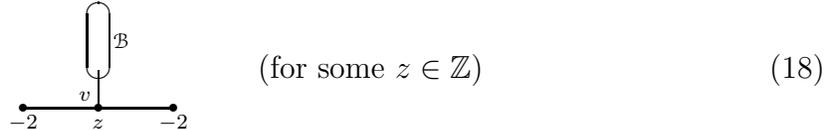
\begin{equation} \label{76g7676f76g76h76d76}
\setlength{\unitlength}{1mm}
\raisebox{-7\unitlength}{\begin{picture}(26,16)(-3,-3)
\put(0,0){\circle*{1}}
\put(10,0){\circle*{1}}
\put(20,0){\circle*{1}}
\put(0,0){\line(1,0){20}}
\put(10,0){\line(0,1){5}}
\put(0,-1){\makebox(0,0)[t]{\tiny $-2$}}
\put(10,-1.5){\makebox(0,0)[t]{\tiny $z$}}
\put(9,1){\makebox(0,0)[br]{\tiny $v$}}
\put(20,-1){\makebox(0,0)[t]{\tiny $-2$}}
\put(12,9){\makebox(0,0)[l]{\tiny $\Beul$}}
\put(10,9){\oval(3,10)}
\end{picture}} \qquad \text{(for some $z \in \Integ$)}
\end{equation}
where $\Beul$ represents a (possibly empty) branch
of $\Geul_n \setminus \{e_n\}$ at $v$;
so the weighted graph~\eqref{76g7676f76g76h76d76} is equivalent to $\emptyset$.
However, \eqref{76g7676f76g76h76d76} is not equivalent to $\emptyset$.
Indeed, if it were, then we would have
$\Beul \sim \emptyset$ and in fact \eqref{76g7676f76g76h76d76}
would contract to
\begin{equation} \label{6476472ghfjdhsjfy347887378}
\setlength{\unitlength}{1mm}
\raisebox{-2\unitlength}{\begin{picture}(26,6)(-3,-3)
\put(0,0){\circle*{1}}
\put(10,0){\circle*{1}}
\put(20,0){\circle*{1}}
\put(0,0){\line(1,0){20}}
\put(0,-1){\makebox(0,0)[t]{\tiny $-2$}}
\put(10,-1.5){\makebox(0,0)[t]{\tiny $t$}}
\put(20,-1){\makebox(0,0)[t]{\tiny $-2$}}
\end{picture}} \qquad \text{(for some $t \in \Integ$)}
\end{equation}
but clearly the graph \eqref{6476472ghfjdhsjfy347887378} is not
equivalent to $\emptyset$.
So \eqref{76g7676f76g76h76d76}
is not equivalent to $\emptyset$ either,
which rules out case~(a).

In case~(b), $\Geul_n \setminus \{e_n\}$ has a connected component as follows:
\begin{equation} \label{7x67x6x6454x35x76x7x87x}
\setlength{\unitlength}{1mm}
\raisebox{-12\unitlength}{\begin{picture}(26,24)(-3,-14)
\put(0,0){\circle*{1}}
\put(10,0){\circle*{1}}
\put(20,0){\circle*{1}}
\put(10,7){\circle*{1}}
\put(0,0){\line(1,0){20}}
\put(10,0){\line(0,1){7}}
\put(0,-1){\makebox(0,0)[t]{\tiny $-2$}}
\put(9,-1){\makebox(0,0)[tr]{\tiny $z$}}
\put(20,-1){\makebox(0,0)[t]{\tiny $-2$}}
\put(10,8){\makebox(0,0)[b]{\tiny $y$}}
\put(10,0){\line(0,-1){5}}
\put(10,-9){\oval(3,10)}
\put(12,-9){\makebox(0,0)[l]{\tiny $\Beul$}}
\end{picture}}
\qquad \text{(for some $z \in \Integ$)}
\end{equation}
where $\Beul$ might be empty and 
all vertices of $\Beul$ have weight strictly less than $-1$.
This implies that the weighted graph~\eqref{7x67x6x6454x35x76x7x87x} is equivalent to the empty graph.
However, \eqref{7x67x6x6454x35x76x7x87x} is not equivalent to $\emptyset$.
Indeed, if it were then we would have $\Beul \sim \emptyset$, so in fact $\Beul = \emptyset$,
then \eqref{7x67x6x6454x35x76x7x87x} would be of the form
\eqref{76g7676f76g76h76d76} and hence would not be equivalent to $\emptyset$.
So \eqref{7x67x6x6454x35x76x7x87x} is not equivalent to $\emptyset$
and case~(b) is ruled out.

Consequently case~(c) must occur, i.e.,
$(\Geul_1, e_1)$ must be
the blowing-up of $(\Geul, v)$ at an edge which does not contain the vertex
of weight $y$.
Note that, although there are two such edges, only one case
needs to be considered because an automorphism of $(\Geul, v)$ interchanges
the two edges.
Also observe that, if the vertex of weight $y$ is called $w$, then
$w$ has the same weight in $\Geul$ and in $\Geul_n$;
consequently $y = -1$, because $\Geul_n \setminus \{e_n\}$ must have
a vertex of weight $-1$ and all vertices of 
$\Geul_n \setminus \{e_n, w \}$ have weight strictly less than $-1$.
So $(\Geul_1, e_1)$ is the following weighted pair:
\begin{equation*} 
\setlength{\unitlength}{1mm}
\raisebox{-5\unitlength}{\begin{picture}(36,13)(-3,-3)
\put(0,0){\circle*{1}}
\put(10,0){\circle*{1}}
\put(20,0){\circle*{1}}
\put(30,0){\circle*{1}}
\put(10,7){\circle*{1}}
\put(0,0){\line(1,0){30}}
\put(10,0){\line(0,1){7}}
\put(0,-1){\makebox(0,0)[t]{\tiny $-2$}}
\put(10,-1){\makebox(0,0)[t]{\tiny $-2$}}
\put(20,1){\makebox(0,0)[b]{\tiny $*$}}
\put(20,-1){\makebox(0,0)[t]{\tiny $-1$}}
\put(30,-1){\makebox(0,0)[t]{\tiny $-3$}}
\put(9,7){\makebox(0,0)[r]{\tiny $-1$}}
\put(11,7){\makebox(0,0)[l]{\tiny $w$}}
\end{picture}}
\end{equation*}
Then $w$ is a contractible vertex and
the contraction of $(\Geul_1, e_1)$ at $w$ is
$(\bar \Geul_1, \bar e_1) = [-3,-1^*, -1, -2]$. Then
$ \ell [-3,-1^*, -1, -2]
= \ell (\bar \Geul_1, \bar e_1)
= \ell (\Geul_1, e_1)
< \ell (\Geul, v) < \infty$,
which implies that $[-3,-1^*, -1, -2]$ is erasable.
This contradicts \ref{g3gg4gg6g8gg9ggg3g3g}, so the proof is complete.
\end{proof}

The next proof requires familiarity with the classical notion of dual graph (see for instance~\ref{difupawejk}).
If $D$ is an SNC-divisor of a nonsingular projective surface $S$,
we write $\Geul(D,S)$ for the dual graph of $D$ in $S$.
Recall in particular that $\Geul(D,S)$ is a weighted graph.
See \ref{ehwgehgwhepcopdopsdopsdo} for the definition of ``chain''.

\begin{proposition} \label {656758765876dfkasbkjgfakh67664}
No triple $(Y_0,D, L)$ satisfies the following conditions \mbox{\rm(i--iii):}
\begin{enumerate}

\item[(i)] $Y_0$ is a nonsingular projective surface and
$D, L \subset Y_0$ are irreducible curves.

\item[(ii)]  $L$ is nonsingular, $L^2 = 0$ and $D \cdot L = 2$.

\item[(iii)] There exists a chain 
$Y_0 \xleftarrow{ \sigma_1 } Y_1 \xleftarrow{ \sigma_2 } \cdots \xleftarrow{ \sigma_N } Y_{N}$
such that $N\ge1$ and, if $D_N \subset Y_N$, $L_N \subset Y_N$, and $G_i \subset Y_N$ denote respectively
the strict transforms of $D$, of $L$, and of the exceptional curve of $\sigma_i$, then:
\begin{itemize}

\item the subset $D_N \cup L_N \cup G_1 \cup \dots \cup G_{N-1}$
of $Y_N$ is the exceptional locus of a birational morphism $Y_N \to S$
where $S$ is a nonsingular projective surface;

\item $L_N^2 \neq -1$ in $Y_N$.

\end{itemize}
\end{enumerate}
\end{proposition}

\begin{proof} 
By contradiction, assume that $(Y_0,D, L)$ exists and consider
$Y_0 \xleftarrow{ \sigma_1 } Y_1 \xleftarrow{ \sigma_2 } \cdots
\xleftarrow{ \sigma_N } Y_{N}$ as in the statement,
where $\sigma_i : Y_i \to Y_{i-1}$ is 
the blowing-up at the point $Q_i \in Y_{i-1}$.
Let $D_i, L_i \subset Y_i$ be the strict transforms of $D_0 = D$ and 
$L_0=L$ respectively;
we write $G_i \subset Y_i$ for the exceptional curve of $\sigma_i$ and,
if $i<j \le N$,
the strict transform of $G_i$ in  $Y_j$ is also denoted by $G_i \subset Y_j$.
For each $i \in \{ 1, \dots, N\}$,
let $\Delta_i$ denote the reduced divisor $D_i + L_i + G_1 + \cdots + G_i$ of $Y_i$.
Let $\Omega$ denote the reduced divisor $D_N + L_N + G_1 + \dots + G_{N-1}$ of $Y_N$,
i.e., $\Omega = \Delta_N - G_N$.

As $\supp(\Omega)$ is the exceptional locus of a birational morphism,
$\Omega$ is an SNC-divisor of $Y_N$ which has at least one $(-1)$-component.
Because $L_N^2 \neq -1$, it follows that $D_N^2 = -1$ and that $D_N$ is the only $(-1)$-component of $\Omega$.
Moreover, there must hold $L_N^2 < -1$ (so $N \ge 2$, $Q_1 \in L_0$ and $Q_2 \in L_1)$.
Also note that
$D_N \cdot L_N \le 1 < 2 = D_0\cdot L_0$, so $Q_1 \in D_0 \cap L_0$.
We record:
\begin{equation} \label{ccuizuizuxiwmw,me.,qw.,qm.q,.,}
\text{$Q_1 \in L_0 \cap D_0$ and $Q_2 \in L_1 \cap G_1$.}
\end{equation}

Suppose that $Q_1$ is a singular point of $D_0$.  
Then $D_0 \cdot L_0 = 2$ implies that $D_1 \cap L_1 = \emptyset$ and that
$D_1 \cdot G_1 = 2$; then \eqref{ccuizuizuxiwmw,me.,qw.,qm.q,.,}
impies that $Q_2 \notin D_1$
and hence that $(D_N \cdot G_1)_{Y_N} = (D_1 \cdot G_1)_{Y_1} > 1$,
which contradicts the
fact that $\Omega$ is an SNC-divisor. This shows that $Q_1$ is a regular point of $D_0$.
As $D_N$ is nonsingular and $\sigma_1 \circ \dots \circ \sigma_N$ is a chain, it follows that $D_0$ is nonsingular.

Consider the case where $D_0 \cap L_0$ is one point (so it is $Q_1$).
Then it follows from \eqref{ccuizuizuxiwmw,me.,qw.,qm.q,.,} that
$\Delta_2 = D_2 + L_2 + G_1 + G_2$ is an SNC-divisor of $Y_2$ whose dual
graph is 
\begin{equation} \label {cqvcqvcavcaqvcqvcavswcvqcvcw}
\setlength{\unitlength}{1mm}
\Geul( \Delta_2, Y_2)\,: \quad
\raisebox{-4\unitlength}{\begin{picture}(26,13)(-3,-3)
\put(0,0){\circle*{1}}
\put(10,0){\circle*{1}}
\put(20,0){\circle*{1}}
\put(10,7){\circle*{1}}
\put(0,0){\line(1,0){20}}
\put(10,0){\line(0,1){7}}
\put(0,-1){\makebox(0,0)[t]{\tiny $-2$}}
\put(10,-1){\makebox(0,0)[t]{\tiny $-1$}}
\put(9,1){\makebox(0,0)[br]{\tiny $*$}}
\put(20,-1){\makebox(0,0)[t]{\tiny $-2$}}
\put(10,8){\makebox(0,0)[b]{\tiny $y$}}
\end{picture}}
\end{equation}
where $y = D_2^2 \in \Integ$ 
and where $G_2$ is the vertex indicated by an asterisk $*$.
Then
$\Delta_i = D_i + L_i + G_1 + \cdots + G_i$ is an SNC-divisor of $Y_i$
for each $i \in \{2, \dots, N\}$, and
$$
\big( \Geul( \Delta_2, Y_2), G_2 \big)
\bup \cdots \bup
\big( \Geul( \Delta_N, Y_N), G_N \big) = ( \Geul, v)
$$
is a sequence of blowings-up of weighted pairs
(cf.~\ref{dfj;wieufajfklaj;dklfj;89868}).
The weighted graph $\Geul \setminus \{v\}$ is equal to $\Geul( \Omega, Y_N )$,
which is equivalent to the empty weighted graph since 
$\supp(\Omega)$ is the exceptional locus of a birational morphism.
So the weighted pair $\big( \Geul( \Delta_2, Y_2), G_2 \big)$ is erasable,
i.e., the weighted pair pictured in \eqref{cqvcqvcavcaqvcqvcavswcvqcvcw} is erasable,
and this contradicts Proposition~\ref{bvbwvwbwvbevbrbtobvbvp}.

\bigskip
This shows that $D_0 \cap L_0$ contains more than one point.
Then it follows from \eqref{ccuizuizuxiwmw,me.,qw.,qm.q,.,} that
$\Delta_1 = D_1 + L_1 + G_1$ is an SNC-divisor of $Y_1$ whose dual
graph is
\begin{equation} \label {83274nf63hfjf}
\setlength{\unitlength}{1mm}
\Geul( \Delta_1, Y_1) \,: \quad
\raisebox{-6\unitlength}{\begin{picture}(26,17)(-3,-3)
\put(0,0){\circle*{1}}
\put(20,0){\circle*{1}}
\put(10,10){\circle*{1}}
\put(10,8.5){\makebox(0,0)[t]{\tiny $*$}}
\put(0,0){\line(1,0){20}}
\put(0,0){\line(1,1){10}}
\put(10,10){\line(1,-1){10}}
\put(0,-1){\makebox(0,0)[t]{\tiny $-1$}}
\put(20,-1){\makebox(0,0)[t]{\tiny $x$}}
\put(10,11){\makebox(0,0)[b]{\tiny $-1$}}
\end{picture}}
\end{equation}
where $x = D_1^2 \ge D_N^2 = -1$ and where $G_1$ is the vertex indicated
by the asterisk.
Then
$$
\big( \Geul( \Delta_1, Y_1), G_1 \big)
\bup \cdots \bup
\big( \Geul( \Delta_N, Y_N), G_N \big) = ( \Geul, v)
$$
is a sequence of blowings-up of weighted pairs such that 
$\Geul \setminus \{v\} = \Geul( \Omega, Y_N ) \sim \emptyset$.
So the weighted pair $\big( \Geul( \Delta_1, Y_1), G_1 \big)$ is erasable,
i.e., the weighted pair pictured in \eqref{83274nf63hfjf} is erasable.
This contradicts Proposition~\ref{bvbwvwbwvbevbrbtobvbvp}, so the proof is complete.
\end{proof}


\section{Existence of a dicritical of degree $1$}
\label {ExistDicDeg1}

\begin{parag}  \label {254desq3sa3235fa}
{\bf Dicriticals.} Let $\Lambda$ be a pencil without fixed components
on a nonsingular projective surface $S$ and
$\Phi_\Lambda : S \dasharrow \proj^1$ the rational map given
by $\Lambda$.
Choose a commutative diagram
\begin{equation}  \label{sjfpq23jq;awsdsa}
\xymatrix{
S  \ar @{-->} [d]_{ \Phi_\Lambda }
&  {\tilde S}  \ar [dl]^{ \Psi_\Lambda } \ar[l]_{ \pi } \\
{\proj^1}
}
\end{equation}
where $\tilde S$ is a nonsingular projective surface,
$\pi$ is a birational morphism and $\Psi_\Lambda$ is a morphism,
and consider the exceptional locus
$\Eeul = \exc(\pi) \subset \tilde S$ of $\pi$.
The horizontal\footnote{A curve $E \subset \tilde S$ is {\it vertical\/}
if $\Psi_\Lambda(E)$ is a point, {\it horizontal\/} otherwise.}
curves included in $\Eeul$ are called the {\it dicriticals\/} of 
diagram~\eqref{sjfpq23jq;awsdsa}.
If $E \subseteq \Eeul$ is a dicritical of \eqref{sjfpq23jq;awsdsa}
then the composition $E \hookrightarrow \tilde S \xrightarrow{\Psi_\Lambda} \proj^1$
is a surjective morphism $f_E : \proj^1 \to \proj^1$; the positive integer 
$\deg( f_E )$ is called the {\it degree of the dicritical $E$}.

Suppose that diagram~\eqref{sjfpq23jq;awsdsa} has $s\ge0$ dicriticals,
of degrees $d_1, \dots, d_s$ respectively. 
Then the number $s$ and the unordered $s$-tuple $[ d_1, \dots, d_s ]$
are uniquely determined by $\Lambda$, i.e., are independent of the choice
of a diagram~\eqref{sjfpq23jq;awsdsa} which resolves the points of
indeterminacy of $\Phi_\Lambda$.
So it makes sense to speak of the number of dicriticals
``of $\Lambda$'', and of the degrees of these dicriticals.
\end{parag}

The main objective of this section is to prove:
 
\begin{theorem}  \label {kjfkjeopozxpozx}
Let $C \subset \proj^2$ be a unicuspidal rational curve
with distinguished point $P$ and let $\Lambda_C$ be the
unique pencil on $\proj^2$ such that $C \in \Lambda_C$
and $\Bs(\Lambda_C) = \{P\}$.
If $C$ is of nonnegative type then $\Lambda_C$ has
either $1$ or $2$ dicriticals, and at least one of them has degree $1$.
\end{theorem}

The fact that $\Lambda_C$ has either one or two dicriticals easily follows
from \ref{dif4938998r984j20w2j}\eqref{hghshsghsghshgs}; the real contents of 
the theorem is the claim that there exists a dicritical of degree $1$.

The proof of the Theorem makes use of Proposition~\ref{656758765876dfkasbkjgfakh67664} 
(see the last sentence of the proof).
The following notation is also needed:

\begin{parag} \label{jcjcjucjcudcdujchc}
Let $(a,b) \in \Integ^2$ be such that $\min(a,b) \ge 1$.
Consider the Euclidean algorithm of $(a,b)$:
\begin{align*}
x_0 &= q_1 x_1 + x_2 \\
&\vdots \\
x_{p-2} &= q_{p-1} x_{p-1} + x_{p} \\
x_{p-1} &= q_p x_p
\end{align*}
where $x_0 = b$, $x_1 = a$, all $x_i$ and $q_i$ are positive integers and
$x_1 > \cdots > x_{p} \ge 1$ (so that $\gcd(a,b) = x_p$).
We define the finite sequence $S(a,b)$ by
$$
S(a,b) = (
\underbrace{x_1, \dots, x_1}_{\text{$q_1$ times}}, \dots, 
\underbrace{x_{p-1}, \dots, x_{p-1}}_{\text{$q_{p-1}$ times}},
\underbrace{x_{p}, \dots, x_{p}}_{\text{$q_{p}$ times}} ).
$$
Note that $S(a,b) = S(b,a)$.
It is well known and easy to verify that if we change the notation to
$S(a,b) = (r_1, r_2,  \dots, r_n)$ then
\begin{equation} \label {hggfqfqgfqqfdqfqddsqds}
\sum_{i=1}^n r_i  = a + b - \gcd(a,b) 
\text{\qquad and \qquad}
\sum_{i=1}^n r_i^2 = ab.
\end{equation}
\end{parag}

\begin{proof}[Proof of Theorem~\ref{kjfkjeopozxpozx}]
Let $C \subset \proj^2$ be a unicuspidal rational curve of nonnegative type, with distinguished point $P$.
Let the notation be as in \ref{sdopfiuqpweijdfklsmdn897823487} and \ref{dif4938998r984j20w2j},
and note that $\Lambda_m$ is a $\proj^1$-ruling by \ref{dif4938998r984j20w2j}(c).
The dicriticals of $\Lambda_C$ are the $E_i \subset S_m$ which are horizontal, i.e.,
which are not included in the support of an element of $\Lambda_m$.
So, by \ref{dif4938998r984j20w2j}\eqref{hghshsghsghshgs}, $\Lambda_C$ has either one or two dicriticals.
To prove that at least one dicritical has degree $1$, we have to
show that there exists $i \in \{ 1, \dots, m \}$
such that $E_i$ is a section of~$\Lambda_m$, i.e., $(E_i \cdot D)_{S_m} = 1$ for all $D \in \Lambda_m$.
Note that $\Lambda_m$ does have a section by Gizatullin's Theorem~\ref{jhdhsgghgaqaqoaiaaqo};
however, we don't know a priori whether a section can be found among the $E_i$.
Proceeding by contradiction, we assume that no $E_i$ is a section
of $\Lambda_m$. As $C_m \in \Lambda_m$ by \ref{dif4938998r984j20w2j}(b), it follows that
\begin{equation} \label{bvbvbvuubviubivubivubiu}
\text{for all $i \in \{ 1, \dots, m \}$, \quad
$E_i \cdot C_m \neq 1$ (in $S_m$)}.
\end{equation}
Then in $S_m$ we have
\begin{equation} \label{dkjfkwjeflkjawlkaskd}
\text{$E_m \cdot C_m > 1$ and for all $i < m$ we have 
$E_i \cap C_m = \emptyset$.}
\end{equation}
Indeed, $E_m \cdot C_m = e_{P_m}(C_{m-1} ) \ge1$ and 
\eqref{bvbvbvuubviubivubivubiu} implies that the inequality is strict.
If for some $i<m$ we have $E_i \cap C_m \neq \emptyset$ then
the fact that $E_i \cap C_m = \{ P_{m+1} \} = E_m \cap C_m$ implies
that $\min( E_i \cdot C_m,  E_m \cdot C_m ) = e_{ P_{m+1}}( C_m ) = 1$,
which contradicts \eqref{bvbvbvuubviubivubivubiu}.
So \eqref{dkjfkwjeflkjawlkaskd} is true.

Consider the multiplicity sequence
$(r_1, \dots, r_m)$ where $r_i = e_{P_i}(C_{i-1}) = (E_i \cdot C_i)_{S_i}$, and note that 
$$
r_m > 1
$$
by the first part of \eqref{dkjfkwjeflkjawlkaskd}.
Let $d = \deg(C)$.  As $C_m^2 = 0$ and $C_m \isom \proj^1$, we have
$0 = C_m^2 = C_0^2 - \sum_{i=1}^m r_i^2 = d^2 - \sum_{i=1}^m r_i^2$
and (by the genus formula) $(d-1)(d-2) = \sum_{i=1}^m r_i(r_i-1)$.
It follows that
\begin{equation} \label {cycuycucycttctcrc4c3c2}
d^2 = \sum_{i=1}^m r_i^2
\quad \text{and} \quad
3d - 2 = \sum_{i=1}^m r_i .
\end{equation}

Note that $(r_1, \dots, r_m)$ cannot be a constant sequence
$(a, \dots, a)$ because equations~\eqref{cycuycucycttctcrc4c3c2} would
then read $d^2 = ma^2$ and $3d-2 = ma$,
and these have no solution in integers with $a>1$.
We point out that $m\ge2$, for otherwise $(r_1, \dots, r_m)$ would be constant.
From the second part of \eqref{dkjfkwjeflkjawlkaskd}
and the fact that $(r_1, \dots, r_m)$ is not constant,
we deduce that $(r_1, \dots, r_m)$ has the following description:
there exist $(a_1, b_1), \dots, (a_h, b_h) \in \Integ^2$
(for some $h \ge 1$) such that
\begin{itemize}

\item $\min( a_i, b_i ) \ge 1$ for all $i \in \{1, \dots, h\}$

\item $a_{i+1} = \gcd(a_i, b_i)$ for all $i \in \{1, \dots, h-1\}$

\item $a_1 > \cdots > a_h > a_{h+1}$,
where we define $a_{h+1} = \gcd(a_h, b_h)$

\item $(r_1, \dots, r_m) =
\big( S(a_1, b_1),\,  \dots,\,  S(a_h, b_h),\, (a_{h+1})_e  \big)$
for some $e\ge0$,
where each sequence $S(a_i, b_i)$ is defined as in \ref{jcjcjucjcudcdujchc}
and where $(a_{h+1})_e$ is the sequence $(a_{h+1}, \dots, a_{h+1} )$ where
$a_{h+1}$ occurs $e$ times.
\end{itemize}
By \ref{jcjcjucjcudcdujchc}, the last term of the sequence $S( a_h, b_h )$
is $\gcd(a_h, b_h) = a_{h+1}$; so $r_m = a_{h+1}$ holds when $e=0$,
and obviously it also holds when $e \neq 0$.
So 
$$
a_{h+1} = r_m > 1
$$
in all cases.
By~\eqref{cycuycucycttctcrc4c3c2} and \eqref{hggfqfqgfqqfdqfqddsqds}, 
$$
d^2 = \sum_{i=1}^h a_i b_i + e a_{h+1}^2
$$
and since $a_{h+1}$ divides each $a_i$ and each $b_i$ it follows that
$a_{h+1}^2 \mid d^2$ and hence that $a_{h+1} \mid d$.
The other part of~\eqref{cycuycucycttctcrc4c3c2} gives
$$
3d-2 = \sum_{i=1}^h (a_i + b_i - a_{i+1}) + e a_{h+1}
= a_1 + (e-1) a_{h+1} + \sum_{i=1}^h b_i ,
$$
so $a_{h+1} \mid 2$ and consequently
\begin{equation} \label{cxfcsfxsdfcxfxd}
r_m = a_{h+1} = 2 .
\end{equation}
Define the integers $\delta = d/2$,
$\alpha_i = a_i/2$ ($1 \le i \le h+1$) and $\beta_i = b_i/2$ ($1 \le i \le h$).
Then $\alpha_{i+1} = \gcd( \alpha_i, \beta_i )$ for
all $i \in \{ 1, \dots, h \}$ and
$\alpha_1 > \cdots > \alpha_h > \alpha_{h+1} = 1$.
The above equations yield:
$$
\delta^2 = \sum_{i=1}^h \alpha_i \beta_i + e
\qquad
3 \delta = \alpha_1 + e + \sum_{i=1}^h \beta_i 
$$
Suppose that $p$ is a prime number which divides both $e$ and $\alpha_h$.
Then $\delta^2 \equiv 0 \pmod{p}$
and $3 \delta \equiv \beta_h \pmod{p}$,
so $p \mid \beta_h$ and consequently
$p \mid \gcd( \alpha_h, \beta_h ) = \alpha_{h+1} = 1$, which is absurd.
This contradiction shows that $\gcd( e, \alpha_h ) = 1$, and since 
$\alpha_h > 1$ we have shown that $e > 0$. This has the following consequence:
\begin{equation} \label{cvcvicoviocpcovocivpco}
\text{the only $i < m$ which satisfies
$E_i \cap E_m \neq \emptyset$ (in $S_m$) is $i = m-1$.}
\end{equation}
As $P_i \in E_{i-1}$ for all $i>1$ (cf.\  \eqref{dedsedwsewdwededswewd}),
we see in particular that $\bigcup_{i=1}^m E_i$ is connected;
by \eqref{cvcvicoviocpcovocivpco},
it follows that the subset $\Eeul = \bigcup_{i=1}^{m-1} E_i$ of $S_m$ is connected.
As each irreducible component of $\Eeul$ is vertical
by \eqref{dkjfkwjeflkjawlkaskd} and \ref{dif4938998r984j20w2j}\eqref{ghghghghtytr},
it follows that
\begin{equation}
\text{$\Eeul \subseteq \supp(F)$ for some $F \in \Lambda_m$}
\end{equation}
because distinct elements of $\Lambda_m$ have disjoint supports.
We claim:
\begin{equation} \label{kdjfkahsdufauwqekjnd}
\text{if $G \in \Lambda_m$ and $G \neq F$ then $G$ is irreducible
and reduced.}
\end{equation}
By contradiction, suppose that $G \in \Lambda_m \setminus \{F\}$ is
not irreducible and reduced. Then the support of $G$ is a union of at least
two curves (otherwise we would have $G=nG_0$ for some $n\ge2$ and some divisor
$G_0$ of $S_m$, and this would contradict the fact ~\ref{jhdhsgghgaqaqoaiaaqo} that $\Lambda_m$ has a section).
Let $L \subset S_m$ be an irreducible component of $G$.
As $E_m$ is horizontal and $\Eeul \subseteq \supp(F)$, $G$ does not contain
any $E_i$, so the image of $L$ in $S_0$
(via $\pi_1 \circ \cdots \circ \pi_m$) is a curve $L_0 \subset S_0$.
As $\emptyset \neq L_0 \cap C \subseteq \Bs( \Lambda_0 ) = \{ P_1 \}$,
we have $P_1 \in L_0$,  
so $L \cap ( \Eeul \cup E_m ) \neq \emptyset$;
as $L \cap \Eeul \subseteq \supp(G) \cap \supp(F) = \emptyset$, 
we have $L \cdot E_m > 0$ (for any irreducible component $L$
of $G$).  As $G \cdot E_m = C_m \cdot E_m = r_m = 2$,
and since $G$ has at least two irreducible components,
it follows that $G = L+M$ where $L,M$ are distinct prime divisors,
$L \cdot E_m = 1 = M \cdot E_m$
and $L \cap \Eeul = \emptyset = M \cap \Eeul$.
Moreover, Gizatullin's Theorem~\ref{jhdhsgghgaqaqoaiaaqo} implies that
$L \isom \proj^1 \isom M$ and that $L^2 = -1 = M^2$.

Let $L_i \subset S_i$ be the strict transform of $L_0$ on $S_i$
and note that $L_m = L$. By the above observations we have
$P_i \in L_{i-1}$ for all $i \in \{1, \dots, m \}$
and $L_m$ satisfies $L_m \isom \proj^1$ and $L_m^2 = -1$.
Define $\mathbf{m}(L_0) = (r_1', \dots, r_m' )$ by
$r_i' = e_{P_i}( L_{i-1} ) = (E_i \cdot L_i)_{S_i}$ and
let us compare $\mathbf{m}(L_0)$ with the sequence
$\mathbf{m}(C) = (r_1, \dots, r_m)$ which we have already considered.
We claim:
\begin{equation} \label {sdsdioidosidudifiisij}
(r_1, \dots, r_m) = 2(r_1', \dots, r_m').
\end{equation}
To see this, note that
\mbox{\scriptsize $
\left( \begin{array}{c}
(E_1 \cdot C_m)_{S_m} \\ \vdots \\ (E_m \cdot C_m)_{S_m}
\end{array} \right)$}
=
\mbox{\scriptsize $
\left( \begin{array}{c}
0 \\ \vdots \\ 0 \\ 2
\end{array} \right)$}
and
\mbox{\scriptsize $
\left( \begin{array}{c}
(E_1 \cdot L_m)_{S_m} \\ \vdots \\ (E_m \cdot L_m)_{S_m}
\end{array} \right)$}
=
\mbox{\scriptsize $
\left( \begin{array}{c}
0 \\ \vdots \\ 0 \\ 1
\end{array} \right)$},
so
\begin{equation}  \label {89r84hqawea}
\mbox{\scriptsize $
\left( \begin{array}{c}
(E_1 \cdot C_m)_{S_m} \\ \vdots \\ (E_m \cdot C_m)_{S_m}
\end{array} \right)$}
= 2
\mbox{\scriptsize $
\left( \begin{array}{c}
(E_1 \cdot L_m)_{S_m} \\ \vdots \\ (E_m \cdot L_m)_{S_m}
\end{array} \right).$}
\end{equation}
Observe that for each $j \in \{ 1, \dots, m \}$ there exists a linear
map $T_j : \Integ^m \to \Integ^j$ which is completely determined by the
values $e_{P_u}( E_v )$ and which has the following property:
given a curve $H_0 \subset S_0$ and its strict transform $H_j$ on $S_j$,
$$
T_j\ :\ \ 
\mbox{\scriptsize $\left( \begin{array}{c}
(E_1 \cdot H_m)_{S_m} \\ \vdots \\ (E_m \cdot H_m)_{S_m}
\end{array} \right)$}
\ \longmapsto\ 
\mbox{\scriptsize $\left( \begin{array}{c}
(E_1 \cdot H_j)_{S_j} \\ \vdots \\ (E_j \cdot H_j)_{S_j}
\end{array} \right)$}.
$$
By \eqref{89r84hqawea} and linearity of $T_j$ it follows that
$(E_i \cdot C_j)_{S_j} = 2 (E_i \cdot L_j)_{S_j}$
for all $i,j$ such that $1 \le i \le j \le m$, so in particular
$r_j = (E_j \cdot C_j)_{S_j} = 2 (E_j \cdot L_j)_{S_j}= 2 r_j'$
for all $j \in \{1, \dots,m\}$.  This proves \eqref{sdsdioidosidudifiisij}.

Let $d' = \deg( L_0 )$.
As $L_m \isom \proj^1$ and $L_m^2 = -1$,
$(d'-1)(d'-2) = \sum_{i=1}^m r_i' ( r_i' - 1 )$
and 
$(d')^2 = \sum_{i=1}^m (r_i')^2 - 1$, so
$3 d' = 1 + \sum_{i=1}^m r_i'$.
Doubling the last equation and using the second part of
\eqref{cycuycucycttctcrc4c3c2} gives
$$
6d' = 2 + \sum_{i=1}^m (2r_i') = 2 + \sum_{i=1}^m r_i = 3d,
$$
so $d = 2d'$.  Then
$$
d^2
= 4 (d')^2
= 4 \sum_{i=1}^m (r_i')^2 - 4 
= \sum_{i=1}^m r_i^2 - 4 
$$
contradicts \eqref{cycuycucycttctcrc4c3c2},
and hence \eqref{kdjfkahsdufauwqekjnd} is proved.

By Gizatullin's result~\ref{jhdhsgghgaqaqoaiaaqo} we may choose a section $\Sigma \subset S_m$
of $\Lambda_m$ and consider the birational
morphism $\rho : S_m \to \nagata$ whose exceptional locus is the union of
the curves in $S_m$ which are $\Lambda_m$-vertical and disjoint
from $\Sigma$.
Recall from the same result~\ref{jhdhsgghgaqaqoaiaaqo} that $\nagata$ is
one of the Nagata-Hirzebruch ruled surfaces and that
$\bbL = \rho_*( \Lambda_m )$ is a base-point-free pencil on $\nagata$
each of whose elements is a projective line.
We have $\exc( \rho ) \subseteq \supp F$ by \eqref{kdjfkahsdufauwqekjnd},
so the number of irreducible components of $\exc( \rho )$
is $1$ less than the number of irreducible components of $\supp F$
(as exactly one component of $F$ meets $\Sigma$).
Recall that the canonical divisors $K_{\proj^2}$ and $K_{\nagata}$ satisfy
$K_{\nagata}^2 = K_{\proj^2}^2 - 1$; so, consideration of 
$$
\proj^2 = S_0 \xleftarrow{\ \ \pi\ \ } S_m \xrightarrow{\ \ \rho\ \ }
\nagata
$$
(where $\pi = \pi_1 \circ \cdots \circ \pi_m$)
shows that $\rho$ contracts exactly $m-1$ curves, and hence that
$F$ has exactly $m$ irreducible components.
As $\Eeul \subseteq \supp(F)$, it follows that $\supp(F) = \Gamma \cup \Eeul$
for some curve $\Gamma \subset S_m$ such that $\Gamma \not\subseteq \Eeul$,
and where we must have $\Gamma^2 = -1$ since no component of $\Eeul$ has
that property.
We have $\Gamma \cap \Sigma = \emptyset$, for otherwise 
~\ref{jhdhsgghgaqaqoaiaaqo} would imply that $F$ has a $(-1)$-component other
than $\Gamma$, which is not the case.
Note that $\Gamma \neq E_m$ since $E_m$ is horizontal,
so $\Gamma$ is not an $E_i$.
It also follows that exactly one element $j \in \{ 1, \dots, m-1 \}$
is such that $\rho( E_j )$ is a curve; in fact $E_j$ is the unique component
of $F$ which meets $\Sigma$ and consequently $\rho( E_j )$ is an element
of $\bbL$.
Let us also observe that $\exc( \rho ) = \Gamma \cup \bigcup_{i \in I} E_i$,
where $I = \{1, \dots, m-1 \} \setminus \{ j \}$, so $\rho(E_m)$ is a curve.

Let us state some properties of the triple $(Y_0, D, L)$,
where we define $Y_0=\nagata$,  $D = \rho( E_m )$ and $L = \rho( E_j )$ 
(the symbol ``$L$'' was used in an earlier part of the proof, but we give it a new meaning here).
Obviously,

\begin{enumerate}

\item[(i)] {\it $Y_0$ is a nonsingular projective surface and
$D, L \subset Y_0$ are irreducible curves.}

\end{enumerate}

We also observe:

\begin{enumerate}

\item[(ii)]  {\it $L$ is nonsingular, $L^2 = 0$ and $D \cdot L = 2$.}

\end{enumerate}

Indeed, we have already noted that $L \in \bbL$,
so $L$ is nonsingular and $L^2=0$.
As $E_m \cdot C_m = 2$ and (since $\exc(\rho) \subseteq \supp F$) $\rho$ is an isomorphism in a neighborhood
of $C_m$, it follows that $D \cdot \rho( C_m ) = 2$;
noting that $\rho( C_m ) \in \bbL$, it follows that
$D \cdot L' = 2$ for any $L' \in \bbL$ and in particular (ii) is true.
Next we note:\footnote{See Definition \ref{ehwgehgwhepcopdopsdopsdo} for the definition of ``chain''.}

\begin{enumerate}

\item[(iii)] {\it There exists a chain 
$Y_0 \xleftarrow{ \sigma_1 } Y_1 \xleftarrow{ \sigma_2 } \cdots \xleftarrow{ \sigma_N } Y_{N}$
such that $N \ge1$ and, if $D_N \subset Y_N$, $L_N \subset Y_N$, and $G_i \subset Y_N$ denote respectively
the strict transforms of $D$, of $L$, and of the exceptional curve of $\sigma_i$, then:
\begin{itemize}

\item the subset $D_N \cup L_N \cup G_1 \cup \dots \cup G_{N-1}$
of $Y_N$ is the exceptional locus of a birational morphism $Y_N \to S$
where $S$ is a nonsingular projective surface;

\item $L_N^2 \neq -1$ in $Y_N$.

\end{itemize}}
\end{enumerate}

This is obtained from $\rho : S_m \to \nagata$ by changing the
notation:  let $N=m-1$ and factor $\rho$ as
$S_m = Y_N \xrightarrow{ \sigma_N } \cdots \xrightarrow{ \sigma_1 } 
Y_0 = \nagata$, where each $\sigma_i$ is a blowing-up at a point.
Just after \eqref{cycuycucycttctcrc4c3c2} we noted that $m \ge2$, so $N\ge1$.
The fact that the blowing-up sequence $(\sigma_1, \dots, \sigma_N)$
is a chain follows from the fact that
$\exc( \rho ) = \Gamma \cup \bigcup_{i \in I} E_i$
(where $I = \{1, \dots, m-1 \} \setminus \{ j \}$)
has exactly one $(-1)$-component.
We have $G_N = \Gamma$, $D_N = E_m$, and $L_N = E_j$, so in particular $L_N^2 \neq -1$.
The subset $D_N \cup L_N \cup G_1 \cup \dots \cup G_{N-1}$
of $Y_N = S_m$ is equal to $\bigcup_{i=1}^m E_i$, which is the exceptional 
locus of the birational morphism $\pi_1 \circ \dots \circ \pi_m$.
So (iii) is true.

By Proposition~\ref{656758765876dfkasbkjgfakh67664},
no triple $(Y_0,D, L)$ satisfies (i--iii).
This contradiction completes the proof of the Theorem.
\end{proof}

\end{document}